\documentclass[12pt]{elsarticle}
\usepackage{epstopdf}
\usepackage[caption=false]{subfig}
\usepackage{multirow}
\usepackage{color}
\usepackage{booktabs}
\usepackage{float}
\usepackage{graphicx}
\usepackage[utf8]{inputenc}
\usepackage{amsmath, amssymb, amsthm}
\usepackage{geometry}
\usepackage{abstract}
\theoremstyle{plain}
\newtheorem{theorem}{Theorem}[section]
\newtheorem{lemma}[theorem]{Lemma}

\newtheorem{proposition}[theorem]{Proposition}
\theoremstyle{definition}
\newtheorem{definition}[theorem]{Definition}

\theoremstyle{remark}
\newtheorem{remark}{Remark}

\providecommand{\keywords}[1]{\textbf{\textit{Keywords:}} #1}
\title{Spectral analysis and sine-transform based preconditioning for a structure-preserving stabilized scheme approximating the space-fractional Allen–Cahn equation with logarithmic potential}
\author[a]{Danyal Ahmad}
\author[b,c]{Stefano Serra-Capizzano}
\author[d]{Muhammad Sohaib}
\author[e]{Cristina~Tablino-Possio}

\address[a]{Faculty of Basic Sciences, GIK Institute of Engineering Sciences \& Technology, Topi, 23640 Swabi, Khyber Pakhtunkhwa, Pakistan}
\address[b]{Dipartimento di Scienza e Alta Tecnologia, Università dell'Insubria, Via Valleggio 11, 22100 Como, Italy}
\address[c]{Department of Information Technology, Uppsala University, Uppsala, Sweden}
\address[d]{Department of Mathematics and Statistics, Bacha Khan University, Charsadda 24461, Pakistan}
\address[e]{Department of Mathematics and Applications, University of Milano - Bicocca, Italy}
\date{22-5-2026}
\begin{document}

\maketitle

\begin{abstract}
We consider an initial boundary value problem of the space fractional Allen-Cahn equation with logarithmic Flory-Huggins potential. As an approximation technique, first-order weighted and shifted Gr\"unwald difference formulae of the left and right fractional derivatives are used.
The main focus of the present work is to study the spectral features of the underlying matrices and matrix-sequences and to design proper preconditioners based on the spectral information. Then a computational and spectral analysis of the resulting preconditioned matrix-sequences is performed. Numerical evidence and a short list of open problems complete the current study.
\end{abstract}
.
\section{Introduction}
Phase-field models form a foundation for the mathematical modelling of interfacial behaviour and phase-separation dynamics in materials science, fluid mechanics, and other applied fields \cite{provatas2010phasefield,ambrosio2008gradient,du2004phase}. The Allen-Cahn equation \cite{allen1979microscopic} is a prototypical model in the category of gradient-flow models. The equation is derived from the minimization of a free-energy functional and is widely used to study the dynamics of phase interfaces under curvature effects. Over the past few years, there has been increasing interest in the notion of fractional extensions of the Allen-Cahn framework, where the standard Laplacian is replaced by a nonlocal fractional operator to allow the inclusion of phenomena of anomalous diffusion and long-range interactions that cannot be modeled by an operator of integer order.

The space-fractional version of the Allen-Cahn equation that includes the Riesz fractional derivative has become a powerful theory to study phase separation in heterogeneous materials that have strong nonlocal interactions. However, the addition of nonlocal operators poses significant challenges in the analytical and numerical complexity, particularly when used to find the free energies of realistic physical systems like the logarithmic Flory-Huggins potentials. When carefully managed, these logarithmic potentials are essential to recover the thermodynamics of phase separation in polymer mixtures and binary mixtures; however, they are highly nonlinear and singular in the immediate vicinity of pure phases, and therefore, developing methods that ensure stability and adherence to maximum-principle behavior is challenging.

To address these issues, recent work has focused on structure-preserving numerical algorithms capable of conserving critical physical invariants, including energy dissipation, maximum principles, and unconditional stability. Zhang and Yang \cite{zhang2025efficient} presented an adaptive, unconditional, maximum-principle-preserving, and energy-stabilizing method for the space-fractional Allen-Cahn (SFAC) equation by treating the logarithmic potential explicitly. Sohaib et al. \cite{sohaib2024sfac} implemented the spectral method to study the effect of fractional order on the interface profile and overall solution dynamics utilizing the SFAC model. Wang and Chen \cite{wang2025finite} studied the SFAC equation with a regularized logarithmic potential function using the finite element method. For temporal discretization, they used the backward Euler's method and proved that the proposed scheme satisfies the discrete energy dissipation property. Chen and Sun \cite{chen2021dimensional} exercised an exponential time differencing Runge-Kutta method for the multidimensional SFAC equation. They used the dimensional splitting approach and computed the matrix exponential associated with only one-dimensional discretized matrices having a Toeplitz structure. Several other techniques also exist in the the literature for numerically solving SFAC equation. The interested reader can find the relevant material in the references therein \cite{he2020spatial,hu2024convergence,cui2022effective,zhang2023adaptive}

In the current work, first we approximate the SFAC equation with logarithmic Flory-Huggins potential by employing first-order weighted and shifted Gr\"unwald difference formulae of the left and right fractional derivatives. Subsequently, we study the spectral features of the underlying matrices and matrix-sequences and we propose proper $\tau$ preconditioners \cite{bini1983spectral} based on the spectral information. Then a computational and spectral analysis of the resulting preconditioned matrix-sequences is performed. In particular, the convergence speed is optimal as the cost per iteration. Furthermore, the numerical experiments demonstrate the effectiveness of the proposed computational framework in capturing the complex dynamics of the SFAC equation. These experiments reveal the significant influence of the fractional order on interface evolution, and phase separation. Moreover, the simulations preserve important physical characteristics, including energy dissipation and boundedness of the solution as shown in the reported figures, thereby confirming the physical consistency of the proposed approach.

Our work is organized as follows. In Section \ref{sec:pb}, we describe the formulation of the problem and in Section \ref{sec:appr} we approximated the given problem by emphasizing the matrix structures arising in the discretization process. Section \ref{sect:symb} is devoted to a structural and spectral analysis of the underlying matrices with emphasis on spectral localization and distribution. Section \ref{sec:tau_preconditioner} contains a short description of the proposed preconditioners with a study of the clustering features of the associated matrix-sequences. Numerical experiments are given in Section \ref{sec:num} and are critically discussed. In the final Section \ref{sec:end}, concluding remarks and open problems are reported.

\section{Problem Definition}\label{sec:pb}
We consider the following initial boundary value problem of the SFAC equation with logarithmic Flory-Huggins potential \cite{zhang2025efficient}

\begin{equation}\label{eq:SFAC}
\begin{cases}
\frac{\partial u}{\partial t}=\epsilon^{2}\mathcal{L}_{\alpha}u-f(u), & x\in\Omega, t\in(0,T], \\
u(x,0)=u_{0}(x), & x\in\Omega, \\
u|_{\partial\Omega}=0, & t\in [0, T],
\end{cases}
\end{equation}
where $\Omega=(a,b)^{d}$ is the bounded domain in $\mathbb{R}^{d}$ ($d=1,2$), $\epsilon$ is a positive constant, and $\mathcal{L}_{\alpha}$, $\alpha \in (1,2)$ is the Riesz fractional derivative operator defined by
\begin{equation}\label{eq:Riesz fractional operator}
\mathcal{L}_{\alpha}u := \frac{\partial^{\alpha}}{\partial|x|^{\alpha}}u = -\frac{1}{2\cos(\frac{\alpha\pi}{2})}\left({}_{a}D_{x}^{\alpha}u(x,t) + {}_{x}D_{b}^{\alpha}u(x,t)\right),
\end{equation}
where ${}_{a}D_{x}^{\alpha}u$ and ${}_{x}D_{b}^{\alpha}u$ represent the left and right Riemann-Liouville fractional derivatives; see \cite{tian2015class} and references therein.

The nonlinear term $f(u)$ is given by the derivative of the logarithmic Flory-Huggins potential $F(u)$
\begin{equation}\label{eq: derivative of logarithmic Flory-Huggins potential}
f(u) = F^{\prime}(u) = \frac{\theta}{2} \ln\left(\frac{1+u}{1-u}\right) - \theta_{c}u,
\end{equation}
where
\begin{equation}\label{eq:logarithmic Flory-Huggins potential}
F(u) = \frac{\theta}{2} [(1+u)\ln(1+u) + (1-u)\ln(1-u)] - \frac{\theta_{c}}{2}u^{2}, \quad 0<\theta<\theta_{c}.
\end{equation}
The graph of the logarithmic potential function and its first order derivative for $\theta_{c} = 1$ and different values of $\theta$ is given in Fig. \ref{fig:Potential_term}. For smaller values of $\theta$ i.e., when $\theta = 0.2, 0.4, 0.6$, $F(\phi)$ has a double-well structure having two symmetric minima at $\phi = \phi_{\alpha}<0$ and $\phi = \phi_{\beta}>0$. These minimas represent stable equilibrium phases, whereas $\phi = 0$ is an unstable equilibrium. This characteristic is observed in the right panel of the same figure by the multiple zeros of $F'(\phi)$ in which the central zero indicates the unstable phase and outer zeros correspond to the stable phases. For larger values of $\theta$, the width of the wells in $F(\phi)$ decreases and the distance between the equilibrium points becomes smaller. When $\theta = 1.0$, $F(\phi)$ losses its double-well structure and becomes convex, having only a single minimum at $\phi = 0$. Accordingly, $F'(\phi)$ becomes monotonic with the origin as a unique zero, representing no phase separation and hence a homogeneous mixed state.
\begin{figure}[ht!]
\centering
\subfloat[\label{subfig:figa_potential}]{\includegraphics[width=0.5\linewidth]{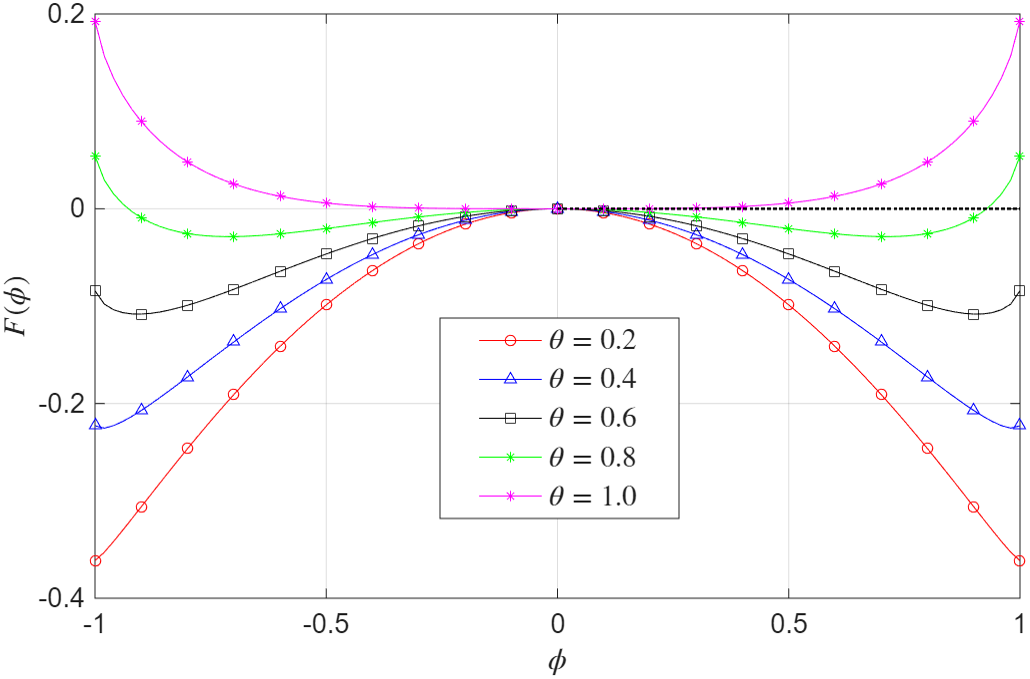}}\hfil
\subfloat[\label{subfig:figb_potential}]{\includegraphics[width=0.5\linewidth]{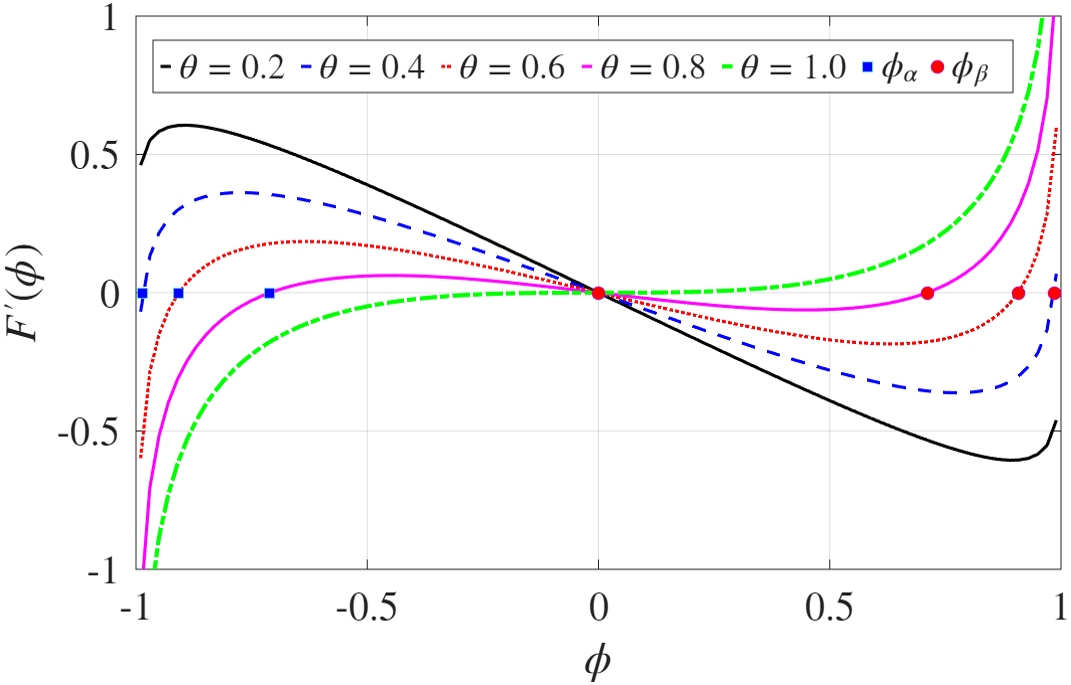}}
\caption{(a) Profile of the logarithmic energy potential $F(\phi)$ when $\theta_{c} = 1$ showing the transition from double well structure to single well as $\theta$ increases (b) Derivative of the logarithmic potential function $F'(\phi)$ for different values of $\theta$ representing stable and unstable equilibrium states.}
\label{fig:Potential_term}
\end{figure}

\section{Numerical Descretization}\label{sec:appr}
We begin reporting the formal definition of the Riemann-Liouville fractional derivatives and their properties following \cite{tian2015class,li2016high}.
\begin{definition}\label{def:RL}
For $\alpha\in (n-1,n)$, the left and right Riemann-Liouville
fractional derivatives of the function $u(x)$ on $[a, b]$ are defined as
    \begin{align}\label{eq:RL}
\begin{split}
    	_{a}{{D}_{x}^{\alpha}}u(x)&=\frac{1}{\Gamma(n-\alpha)}\frac{d^{n}}{dx^{n}}\int_{a}^{x}\frac{u(\xi)}{(\,x-\xi)\,^{\alpha-n+1}}d{\xi},\\
	_{x}{{D}_{b}^{\alpha}}u(x)&=\frac{1}{\Gamma(n-\alpha)}\frac{d^{n}}{dx^{n}}\int_{x}^{b}\frac{u(\xi)}{(\,\xi-x)\,^{\alpha-n+1}}d{\xi}.
\end{split}
\end{align}
\end{definition}
\begin{theorem}\label{Theorem1}
\cite{li2016high} Let $u(x)\in L^{1}(\mathbb{R})$, $_{-\infty}{\mathbb{\boldsymbol{D}}_{x}^{\alpha+2}}u(x)$, $_{x}{\mathbb{\boldsymbol{D}}_{\infty}^{\alpha+2}}u(x)$ and its Fourier transform belong to $L^{1}(\mathbb{R})$.
We define the left and right weighted and shifted Gr\"unwald difference ($\mbox{WSGD}$) operator by
	\begin{equation}\label{eq:aWSGO}
		{_{L}{D}_{h,p,q}^{\alpha}}u(x)=\frac{\alpha-2q}{2(p-q)}	{\mathcal{A}_{h,p}^{\alpha}}u(x)+\frac{2p-\alpha}{2(p-q)}	{\mathcal{A}_{h,q}^{\alpha}}u(x),
	\end{equation}
		and
	\begin{equation}\label{eq:bWSGO}
		{_{R}{D}_{h,p,q}^{\alpha}}u(x)=	\frac{\alpha-2q}{2(p-q)}{\mathcal{B}_{h,p}^{\alpha}}u(x)+\frac{2p-\alpha}{2(p-q)}{\mathcal{B}_{h,q}^{\alpha}}u(x),
	\end{equation}
 then we have
\begin{align*}
    \begin{split}
                {_{L}{D}_{h,p,q}^{\alpha}}u(x)&=_{-\infty}{\mathbb{\boldsymbol{D}}_{x}^{\alpha}}u(x)+O(\,h^{2})\,,\\
       {_{R}{D}_{h,p,q}^{\alpha}}u(x)&=_{x}{\mathbb{\boldsymbol{D}}_{\infty}^{\alpha}}u(x)+O(\,h^{2})\,,
    \end{split}
\end{align*}
uniformly for $x\in\mathbb{R}$, where $p,~q$ are integers and $p\neq q$. where the operators ${\mathcal{A}_{h,p}^{\alpha}}u(x)$ and ${\mathcal{B}_{h,p}^{\alpha}}u(x)$ are the first-order WSGD approximations of the left and right fractional derivatives in \eqref{eq:RL}.
\end{theorem}

More precisely, let $h_{x}={(b-a)}/{(M+1)}$ represents the space step and let $\Omega_{h}=\{x_{i}|x_{i}=a+ih_{x}, 0\le i\le M+1\}$  the spatial grid partition. The first-order WSGD approximation are given by
\begin{align*}
{{}_{a}D_{x}^{\alpha}u(x_{i},t)} & = \mathcal{A}_{h,p}^{\alpha}u(x) + O(h_{x}) = {\frac{1}{h_{x}^{\alpha}} \sum_{k=0}^{i+1} g_{k}^{(\alpha)}u(x_{i-k+1},t) + O(h_{x})}, \\
{{}_{x}D_{b}^{\alpha}u(x_{i},t)} & = \mathcal{B}_{h,p}^{\alpha}u(x) + O(h_{x}) = {\frac{1}{h_{x}^{\alpha}} \sum_{k=0}^{M-i+1} g_{k}^{(\alpha)}u(x_{i+k-1},t) + O(h_{x})},
\end{align*}
where $g_k^{(\alpha)}$ are the standard Gr\"{u}nwald-Letnikov coefficients defined as
\begin{equation}
    g_0^{(\alpha)} = 1, \quad g_k^{(\alpha)} = \left( 1 - \frac{\alpha+1}{k} \right) g_{k-1}^{(\alpha)}, \quad k=1, 2, \dots.
\end{equation}

Notice that for $(p,q)=(0,-1)$, the approximation turns out to be unstable in the case of time-dependent problems, thus we considered the two sets $(p,q)=(1,0)$ and $(1,-1)$ \cite{tian2015class}.
\par
Therefore, in our setting, the compact form of \eqref{eq:aWSGO}-\eqref{eq:bWSGO} for Riemann-Liouville fractional derivatives on $x_{i}$ are

\begin{equation}
{}_{a}D_{x}^{\alpha}u(x_{i},t) = \frac{1}{h_{x}^{\alpha}} \sum_{k=0}^{i+1} \omega_{k}^{(\alpha)}u(x_{i-k+1},t) + O(h_{x}^{2}) \tag{2.3}
\end{equation}
\begin{equation}
{}_{x}D_{b}^{\alpha}u(x_{i},t) = \frac{1}{h_{x}^{\alpha}} \sum_{k=0}^{M-i+1} \omega_{k}^{(\alpha)}u(x_{i+k-1},t) + O(h_{x}^{2}) \tag{2.4}
\end{equation}
where the weight coefficients $\omega_{k}^{(\alpha)}$ satisfy
\begin{equation}
    \begin{split}
        (p,q)&=(1,0), \quad w_k^{(\alpha)} =
    \begin{cases}
    \frac{\alpha}{2} g_0^{(\alpha)}, & k = 0, \\
    \frac{\alpha}{2} g_k^{(\alpha)} + \frac{2-\alpha}{2} g_{k-1}^{(\alpha)}, & k \geq 1,
    \end{cases}\\
        (p,q)&=(1,-1),\quad w_k^{(\alpha)} =
    \begin{cases}
    \frac{\alpha+2}{4} g_k^{(\alpha)}, & k = 0,1, \\
    \frac{\alpha+2}{4} g_k^{(\alpha)} + \frac{2-\alpha}{4} g_{k-2}^{(\alpha)}, & k \geq 2.
    \end{cases}
    \end{split}
\end{equation}
To sum up, the discrete Riesz fractional operator is approximated as
\begin{equation}
\mathcal{L}_{\alpha}u(x_{i},t) = -\frac{1}{2h_{x}^{\alpha}\cos(\frac{\alpha\pi}{2})} \left[ \sum_{k=0}^{i+1} \omega_{k}^{(\alpha)}u(x_{i-k+1},t) + \sum_{k=0}^{M-i+1} \omega_{k}^{(\alpha)}u(x_{i+k-1},t) \right] \tag{2.7}
\end{equation}
Now, let $h_t={T}/{N}$ represents the time step and let $\Omega_{h_t}=\{t_{n}|t_{n}=n h_t, 0\le n\le N\}$ the time grid partition. Using a first-order method for time and adding a stabilized term $s(u^{n+1}-u^{n})$ with $s>0$, the fully discrete scheme is
\begin{equation}
\frac{u^{n+1}-u^{n}}{dt} + s(u^{n+1}-u^{n}) = \epsilon^{2}\mathcal{L}_{\alpha}u^{n+1} - f(u^{n}), \quad 0 \le n \le N-1 \tag{2.12}
\end{equation}
In matrix form, this is expressed as
\begin{equation}\label{eq:matrix_form}
\left((1+sdt)I - \epsilon^{2}dt \mathcal{T}^{(\alpha)}_{M}\right) u^{n+1} = (1+sdt)u^{n} - dt f(u^{n})
\end{equation}
where
\[ T_{\alpha,M}=\overline{c}_{\alpha} G_{\alpha,M} \]
with
\begin{align}
\overline{c}_{\alpha} & =-\left({2h_{x}^{\alpha}\cos({\alpha\pi}/{2})}\right)^{-1},\\
        G_{\alpha,M} & ={\mathcal{G}_{\alpha,M}+\mathcal{G}_{\alpha,M}^{T}},\label{eq:BM} \\
\mathcal{G}_{\alpha,M} & = \begin{bmatrix}
\omega_1^{(\alpha)} & \omega_0^{(\alpha)} & 0& \cdots & 0 \\
\omega_2^{(\alpha)} & \omega_1^{(\alpha)} & \omega_0^{(\alpha)}& \cdots &0 \\
\vdots & \ddots & \ddots & \ddots & \vdots \\
\vdots& \ddots & \ddots & \ddots &\omega_0^{(\alpha)}\\
\omega_{M}^{(\alpha)} & \omega_{M-1}^{(\alpha)} & \cdots &\omega_2^{(\alpha)} & \omega_1^{(\alpha)}
\end{bmatrix}. \label{eq:Toeplitz matrix}
\end{align}
	By defining
	\begin{align}\label{eq:Coe_Matrix_form_1D}
		A_{\alpha,M}=\left(1+s h_t\right)I_{M} - \epsilon^{2} h_t{T}_{\alpha,M}=\left(1+s h_t\right)I_{M} - \epsilon^{2}h_t \overline{c}_{\alpha}{G}_{\alpha,M},
	\end{align}
the final linear system of \eqref{eq:matrix_form} becomes
\begin{equation}\label{eq:Final_Linearsystem}
A_{\alpha,M} u^{(n+1)} = (1+s h_t)u^{(n)} - h_t f(u^{(n)}).
\end{equation}
The following Lemma \ref{Lemma1} provides useful properties of the weight coefficients allowing to classify the Toeplitz matrix symbol in the next section..
\begin{lemma}\label{Lemma1}
\cite{tian2015class} For $\alpha \in (1,2)$, it holds
\begin{enumerate}
    \item if $(p,q)=(1,0)$, we have
    \begin{align*}
    	    &\omega_{0}^{(\alpha)}=\frac{\alpha}{2},~\omega_{1}^{(\alpha)}=\frac{2-\alpha-\alpha^2}{2}<0,~\omega_{2}^{(\alpha)}=\frac{\alpha\left(\alpha^2+\alpha-4\right)}{4},\\
            & 1\geq{\omega_{0}^{(\alpha)}\geq\omega_{3}^{(\alpha)}\geq\cdots\geq0},\\
            &\sum_{k=0}^{m}{\omega_{k}^{(\alpha)}}<0,~m\geq2, \quad\sum_{k=0}^{\infty}{\omega_{k}^{(\alpha)}}=0.
\end{align*}	
    \item if $(p,q)=(1,-1)$, we have
    \begin{align*}
    	    &\omega_{0}^{(\alpha)}=\frac{\alpha+2}{4},~\omega_{1}^{(\alpha)}=-\frac{\alpha\left(2+\alpha\right)}{4}<0,~\omega_{2}^{(\alpha)}=\frac{\left(\alpha^3+\alpha^2-4\alpha+4\right)}{8}>0,\\
            &\omega_{3}^{(\alpha)}=\frac{\alpha\left(2-\alpha\right)\left(\alpha^2+\alpha-8\right)}{6}<0,\\
            & 1\geq{\omega_{0}^{(\alpha)}\geq\omega_{2}^{(\alpha)}\geq\omega_{4}^{(\alpha)}\geq\cdots\geq0},\\
            &\sum_{k=0}^{m}{\omega_{k}^{(\alpha)}}<0,~m=1~\mbox{or}~ m\geq3, \quad\sum_{k=0}^{\infty}{\omega_{k}^{(\alpha)}}=0.
\end{align*}	
\end{enumerate}
\end{lemma}
For two-dimensional case, let us fix $M_1,M_2,N\in\mathbb{N}$, $\textbf{m} = (M_1,M_2)$, $M=M_{1}M_{2}$ and discretize the domain $\Omega\times[0,T]$ with
\begin{equation*}\label{3D_GRID}
\begin{array}{lllll}
x_{i_1}=a_1+ih_{x_{1}},		&& h_{x_{1}}=\displaystyle \frac{b_1-a_1}{M_1+1}, 		&& i_1=0,...,M_1+1,  	\\
x_{i_2}=a_2+ih_{x_{2}},	    && h_{x_{2}}=\displaystyle \frac{b_2-a_2}{M_2+1},  		&& i_2=0,...,M_2+1.  	
\end{array}
\end{equation*}
The coefficient matrix is defined as
	\begin{align}\label{eq:Coe_Matrix_form_2D}
		{{A}_{\boldsymbol{\alpha},\mathbf{m}}}=\left(1+s h_t\right)I_{\mathbf{m}} - \epsilon^{2} h_t T_{\boldsymbol{\alpha},\mathbf{m}},
	\end{align}
 where $T_{\boldsymbol{\alpha},\mathbf{m}}$, $\boldsymbol{\alpha} = (\alpha_1,\alpha_2)$, denotes the twolevel Toeplitz matrix defined as
\begin{equation}\label{eq:BTTB_matrix}
T_{\boldsymbol{\alpha},\mathbf{m}}=  \overline{c}_{\alpha_1} T_{\alpha_1,M_{1}} \otimes I_{M_2}
+
I_{M_1} \otimes \overline{c}_{\alpha_2} T_{\alpha_2,M_{2}}.
\end{equation}
The final linear system becomes
\begin{equation}\label{eq:Final_Linearsystem_2d}
{{A}_{\boldsymbol{\alpha},\mathbf{m}}} u^{(n+1)} = (1+sh_t)u^{(n)} - h_t f(u^{(n)}).
\end{equation}
\section{ Spectral analysis of the coefficient matrix}\label{sect:symb}
	This section is devoted to the study of the spectral properties of the coefficient matrix-sequence $\{A_{\alpha,M}\}_{M\in\mathbb{N}}$, which is a well-known Toeplitz sequence. We then determine its generating function and study its spectral distribution using spectral tools for Toeplitz sequences. To this aim, let us first introduce some basic definitions and results related to the generating function of a Toeplitz sequence.
    \subsection{Notation and terminology}
	\begin{definition}\label{def3}
	\cite{chan1991toeplitz} Let  $T_{M}\in\mathbb{C}^{M\times M}$ be the Toeplitz matrix of the form
	\begin{equation}\label{eq:Toeplitz_M}
	    T_{M}=\begin{bmatrix}
		& b_{0} & b_{-1} & b_{-2} & \cdots & \cdots & b_{1-M} &\\
		& b_{1} & b_{0} & b_{-1} & \cdots & \cdots & b_{2-M} &\\
		& \vdots & \ddots & \ddots & \ddots &\ddots & \vdots&\\
		& \vdots & \ddots & \ddots &\ddots &\ddots & \vdots&\\
		& b_{M-2} & \ddots & \ddots & \ddots & \ddots & b_{-1} &\\
		& b_{M-1} & b_{M-2} &\cdots & \cdots & b_{1} & b_{0} &
	\end{bmatrix}
	\end{equation}	
	with \begin{equation}\label{eq:fourier_coef}
		b_{k}=\frac{1}{2\pi}\int_{-\pi}^{\pi}{f(x){\rm{e}}^{-\iota kx}dx},\quad\iota^{2}=-1,~~k\in\mathbb{Z},
	\end{equation}
	the Fourier coefficients of a function $f\in L^1(-\pi,\pi)$.
	Then the Toeplitz sequence $\left\{{T_{M}}\right\}_{M\in \mathbb{N}}$ 
	is called the sequence of Toeplitz matrices generated by $f$ and the matrix $T_{M}$ in \eqref{eq:Toeplitz_M} is denoted by $T_{M}\left(f\right)$. The function $f$ is called the generating function, both of the whole sequence of matrices and of the single matrix $T_{M}\left(f\right)$.
	\end{definition}
	
	Note that given a Toeplitz matrix $T_{M}$ as in \eqref{eq:Toeplitz_M}, in order to have a generating function associated with the Toeplitz sequence, we need that there exists $f\in L^1(-\pi,\pi)$ for which the relationship (\ref{eq:fourier_coef}) holds for every $k\in\mathbb{Z}$. In the case where the partial Fourier sum
	\[
	\sum_{k=-M+1}^{M-1}{b_{k}{\rm{e}}^{\iota kx}}
	\]
	converges to $f$ when $M\to \infty$ in infinity norm, then $f$ is a continuous $2\pi$ periodic function given the Banach structure of this space. A sufficient condition is that $\sum_{k=-\infty}^{\infty}\vert{b_{k}}\vert<\infty$, i.e., the generating function belongs to the Wiener class, which is a closed sub-algebra of the continuous $2\pi$ periodic functions.
\begin{definition}\label{def5}
	Let $f:D\rightarrow\mathbb{C}$ be a measurable function defined over a measurable set $D\subset \mathbb{R}^d$, $d\ge 1$, having positive and finite measure $\mu_d(D)\in (0,\infty)$. Let $\mathcal{C}_{0}(\mathbb{K})$  be the set of continuous functions with compact support over $\mathbb{K}\in\{\mathbb{C},\mathbb{R}\}$ and let $\{A_{M}\}_{M\in\mathbb{N}}$ be a sequence of matrices of size $M$ where the eigenvalues of $A_M$ are denoted by $\mu_{j}(A_{M}),~j=1,2,\cdots,M$ and the singular values of $A_M$ are denoted by $\sigma_{j}(A_{M}),~j=1,2,\cdots,M$.
	
We say that $\{A_{M}\}_{M\in\mathbb{N}}$ is distributed as the pair $\left(f,D\right)$ in the sense of the eigenvalues, and write
    \begin{align*}
        \{A_{M}\}_{M\in\mathbb{N}}\sim_\lambda\left(f,D\right),
    \end{align*}
    if the following relation holds for all $F\in\mathcal{C}_{0}(\mathbb{C})$:
    \begin{align}\label{eq:Symbol_def}
        \lim_{M\rightarrow\infty}\frac{1}{M}\sum_{j=1}^{M}F\left(\mu_{j}(A_{M})\right)=\frac{1}{\mu_d(D)}\int_{D} F(f(t))dt.
    \end{align}
    Thus, we write that $f$ is the (spectral) symbol of the matrix-sequence $\{A_{M}\}_{M\in\mathbb{N}}$.

We say that $\{A_{M}\}_{M\in\mathbb{N}}$ is distributed as the pair $\left(f,D\right)$ in the sense of the singular values, and write
    \begin{align*}
        \{A_{M}\}_{M\in\mathbb{N}}\sim_\sigma\left(f,D\right),
    \end{align*}
    if the following relation holds for all $F\in\mathcal{C}_{0}(\mathbb{R})$:
    \begin{align}\label{eq:sv-Symbol_def}
        \lim_{M\rightarrow\infty}\frac{1}{M}\sum_{j=1}^{M}F\left(\sigma_{j}(A_{M})\right)=\frac{1}{\mu_d(D)}\int_{D} F(|f(t)|)dt.
    \end{align}
    Thus, we write that $f$ is the singular value symbol of the matrix-sequence $\{A_{M}\}_{M\in\mathbb{N}}$.
\end{definition}

\begin{remark}\label{remark2}
When $f$ is continuous and $D=[a,b]$, an informal interpretation of \eqref{eq:Symbol_def} is that when the matrix size is sufficiently large, the eigenvalues of $A_{M}$ can be approximated by a sampling of $f$ on a uniform equispaced grid of $[a,b]$, up to few outliers (at most $o(M)$).
\end{remark}
Similar comments can be done for $D\subset \mathbb{R}^d$, $d\ge 1$, as reported in \cite[Remark 2]{DMS-2016} and in \cite[Remark 2.7]{MRS-2019} also in more generality, while for the singular values in relation (\ref{eq:sv-Symbol_def}) again the same statement can be repeated verbatim.

For Toeplitz matrix-sequences, the following result holds.

\begin{theorem}\label{theorem2}
Let $f\in L^{1}([-\pi,\pi])$. Then,
\begin{align*}
        \{T_{M}(f)\}_{M\in\mathbb{N}}\sim_\sigma\left(f,[-\pi,\pi]\right),
\end{align*}
that is, the generating function of the sequence is also the singular value symbol.
Furthermore, if $f$ is real-valued function almost everywhere (a.e.), hen
\begin{align*}
        \{T_{M}(f)\}_{M\in\mathbb{N}}\sim_\lambda\left(f,[-\pi,\pi]\right),
\end{align*}
that is, the generating function of the sequence is also the (spectral) symbol.
\end{theorem}

\begin{remark}\label{rem:szego results}
The previous theorem was originally proven by Szeg\H{o} \cite{grenander1984toeplitz} using analytic tools and the theory of orthogonal polynomials.
The generalizations to the $L^1$, together with their multilevel and block versions, can be found in \cite{ZaTy-L1,Tilli-L1} with the use of genuine matrix-theoretic tools; see also \cite[Chapter 6, Section 10.1]{GLT-I} and  \cite[Chapter 3]{GLT-II} for more results including those of localization nature.

In particular, for a real-valued generating function with essinf$f<$esssup$f$, it holds that any eigenvalue of $T_M(f)$ lies in the open interval $($essinf$f,$esssup$f)$ and hence, with reference to Remark \ref{remark2}, there are no outliers.
The same type of observations carries over $d$-level Toeplitz matrices having generating functions which are real-valued a.e. on their natural definition domain $[-\pi,\pi]^d$, see again \cite[Chapter 3]{GLT-II}.
\end{remark}


\subsection{Analysis of the space and space-time matrix-sequences}

\begin{proposition}\label{Proposition1} Let $\alpha\in (1, 2)$. The generating function associated to the matrix-sequence $\left\{\mathcal{G}_{\alpha,M}\right\}_{M\in\mathbb{N}}$, where $\mathcal{G}_{\alpha,M}$ is defined in \eqref{eq:Toeplitz matrix}, belongs to the Wiener class.
\end{proposition}
\begin{proof}
For the case when $(p,q)=(1,0)$, let us observe that $\mathcal{G}_{\alpha,M}$  is defined by the trigonometric polynomial
	\begin{equation*}
		{b_{\alpha,M}}\left(x\right)=\sum_{k=0}^{M}{\omega_{k}^{(\alpha)}{\rm{e}}^{\iota (k-1)x}},
	\end{equation*}
which converges to ${b_{\alpha}}\left(x\right)=\sum_{k=0}^{\infty}{\omega_{k}^{(\alpha)}{\rm{e}}^{\iota (k-1)x}}$ so that
${b_{\alpha}}$ is well defined and it is the generating function of the matrix-sequence $\left\{\mathcal{G}_{\alpha,M}\right\}_{M\in\mathbb{N}}$. To prove this we prove that $\sum_{k=0}^{\infty}\vert \omega_{k}^{(\alpha)} \vert<\infty$, which is equivalent to demonstrate that ${b_{\alpha,M}}$ lies in Wiener class for $\alpha\in \left(1,2\right)$.
From Lemma \ref{Lemma1}, we know that $\omega_{0}^{(\alpha)} >0$, $\omega_{1}^{(\alpha)} <0$, and $\omega_{k}^{(\alpha)} \ge 0$ for $k\ge 3$.
In addition, though $\omega_{2}^{(\alpha)}$ changes in sign, $\omega_{1}^{(\alpha)}+\omega_{2}^{(\alpha)} <0$ for $\alpha \in (1,2)$.
Then,
\begin{equation*}
\sum_{k=0}^{\infty}\vert \omega_{k}^{(\alpha)} \vert=\sum_{\substack{k=0 \\ k\neq 1,2}}^{\infty}{\omega_{k}^{(\alpha)}}+\vert\omega_{1}^{(\alpha)}\vert+\vert \omega_{2}^{(\alpha)}\vert.
\end{equation*}
By Lemma \ref{Lemma1}, we have also
\begin{equation*}
	\sum_{k=0}^{\infty}\omega_{k}^{(\alpha)} =0,
\end{equation*}
hence
\begin{equation*}
\sum_{\substack{k=0 \\ k\neq 1,2}}^{\infty}{\omega_{k}^{(\alpha)}} = \vert\omega_{1}^{(\alpha)}+ \omega_{2}^{(\alpha)}\vert
\end{equation*}
and
\begin{equation*}
\sum_{k=0}^{\infty}\vert \omega_{k}^{(\alpha)} \vert = \vert \omega_{1}^{(\alpha)} \vert + \vert \omega_{2}^{(\alpha)} \vert + \vert \omega_{1}^{(\alpha)}+ \omega_{2}^{(\alpha)} \vert   \le 2 (\vert \omega_{1}^{(\alpha)} \vert + \vert \omega_{2}^{(\alpha)} \vert),
\end{equation*}
which implies that $b_{\alpha}$ belongs to the Wiener class. Similarly, we can prove the claim for the case where $(p,q)=(1,-1)$.
\end{proof}

After defining \eqref{eq:BM}, we introduce
	\begin{equation*}
		f_{\alpha,M}\left(x\right)={	b_{\alpha,M}\left(x\right)+\overline{	b_{\alpha,M}\left(x\right)}}=2
		\sum_{k=0}^{M}{\omega_{k}^{(\alpha)}\cos\left((k-1)x\right)},
	\end{equation*}
which clearly converges to $f_{\alpha}=b_{\alpha}+\overline{b_{\alpha}}$ according to Proposition \ref{Proposition1}.

In the next proposition, we provide the generating function of the Toeplitz sequence $\left\{G_{\alpha,M}\right\}_{M\in\mathbb{N}}$.

\begin{proposition}\label{Proposition2}
\cite{tian2015class} The matrix-sequence $\left\{G_{\alpha,M}\right\}_{M\in\mathbb{N}}$ is generated by the function
\begin{equation}\label{eq:symf}
	f_{\alpha}\left(x\right)=\begin{cases}
	    2^{\alpha+1}\Big(\sin{\frac{x}{2}}\Big)^{\alpha}\Big(\frac{\alpha}{2}\cos{\left(\frac{\alpha}{2}(x-\pi)-x\right)}+\frac{2-\alpha}{2}\cos{\left(\frac{\alpha}{2}(x-\pi)\right)}\Big),\quad (p,q)=(1,0),\\
        \\
        2^{\alpha+1}\Big(\sin{\frac{x}{2}}\Big)^{\alpha}\Big(\frac{\alpha}{2}\sin{\left(\frac{\alpha}{2}(x-\pi)-x\right)}\sin(x)+\cos{\left(\frac{\alpha}{2}(x-\pi)\right)\cos(x)}\Big),\quad (p,q)=(1,-1).
	\end{cases}
\end{equation}
\end{proposition}
\begin{remark}\label{remark3}
	It is worth noticing that the function $f_\alpha(x)$ has a zero at the origin and is negative for $x\neq 0$.
	\end{remark}
\begin{figure}
		\centering
		\begin{subfloat}[$(p,q)=(1,0).$]
			{\resizebox*{7cm}{!}{\includegraphics[width=\textwidth]{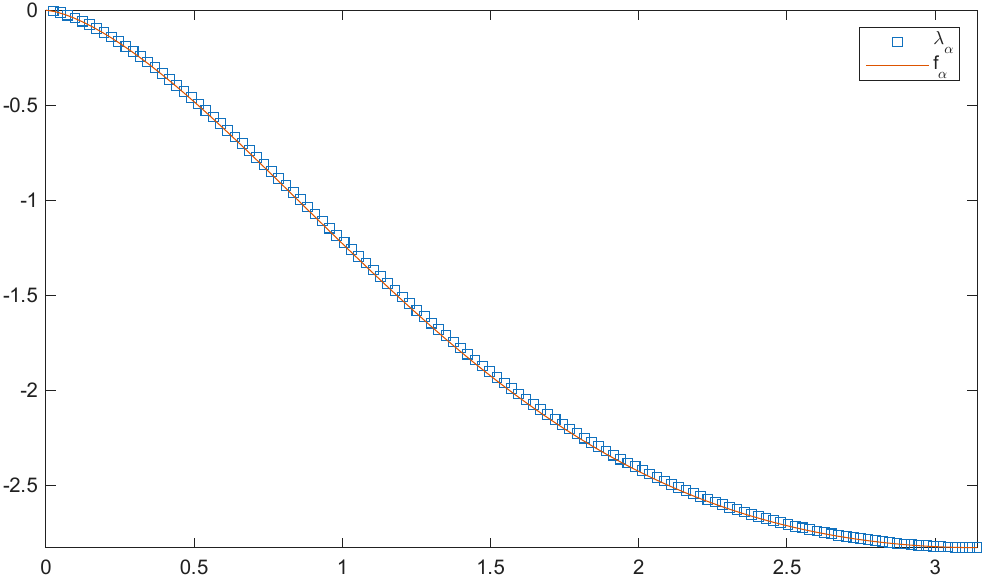}}}
		\end{subfloat}
		\begin{subfloat}[$(p,q)=(1,-1).$]
			{\resizebox*{7cm}{!}{\includegraphics[width=\textwidth]{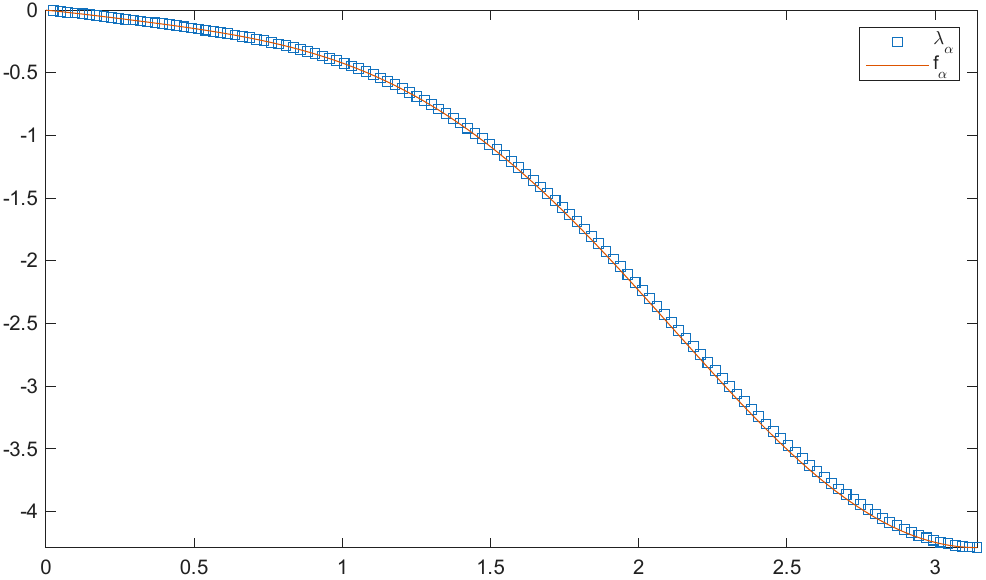}}}
		\end{subfloat}
		\caption{Graph of the symbol $f_{\alpha}\left(x\right)$ and the eigenvalues of the Toeplitz sequence $\left\{G_{\alpha,M}\right\}_{M\in\mathbb{N}}$ for $\alpha=1.5$, and $M=2^{7}$.}
	\label{fig: asymptotical behaviour}
	\end{figure}
	
\begin{theorem}\label{Proposition3_2D_1}
Let $\gamma = \gamma(h_t,h_x) = - {h_t}/{( 2h_{x}^{\alpha} \cos({\alpha\pi}/{2}))}$ such that
$\lim_{h_t,h_x \rightarrow 0^+} \gamma(h_t,h_x)=\ \gamma^*<\infty$.
Then, the symmetric positive definite matrix-sequence is such that
\begin{equation}\label{th1-1D}
\{ A_{\alpha,M}  \}_M  \sim_{\lambda}  1 - \epsilon^2 \gamma^* f_\alpha(x).
\end{equation}
In the case where the sequence $\gamma(h_t,h_x)$ is unbounded with limit $\gamma^*=\infty$, the scaled matrix-sequence
$\{\hat A_{\alpha,M}\equiv \gamma^{-1}(h_t,h_x) A_{\alpha,M}  \}_M$ is such that
\begin{equation}\label{th2-1D}
\{\hat A_{\alpha,M}  \}_M  \sim_{\lambda}  - \epsilon^2  f_\alpha(x).
\end{equation}
Similarly, in the two-dimensional setting, let $\gamma_j = \gamma(h_t,h_{x_j}) = - {h_t}/{( 2h_{x_j}^{\alpha_j} \cos({\alpha_j\pi}/{2}))}$, $j=1,2$,
such that
\[
\lim_{h_t,h_{x_j} \rightarrow 0^+} \gamma_j(h_t,h_x)=\ \gamma_j^* < \infty,\ j=1,2.
\]
Then, the symmetric positive definite matrix-sequence is such that
\begin{equation}\label{th1-2D}
\{ A_{\boldsymbol{\alpha},\mathbf{m}}  \}_\mathbf{m}  \sim_{\lambda}  1 - \epsilon^2 ( \gamma_1^* f_{\alpha_1}(x_1)+ \gamma_2^* f_{\alpha_2}(x_2)),
\end{equation}
where necessarily $\gamma_2^*=0$ if $\alpha_1>\alpha_2$, and alternatively $\gamma_1^*=0$ if $\alpha_2>\alpha_1$.
In the case where the sequence $\gamma=\max_j\gamma_j=\gamma_{\hat j}$ is unbounded with limit $\gamma^*=\infty$, $\hat j=1$ or $\hat j=2$, the scaled matrix-sequence  $\{\hat A_{\boldsymbol{\alpha},\mathbf{m}}\equiv \gamma^{-1} A_{\boldsymbol{\alpha},\mathbf{m}}  \}_\mathbf{m}$ is such that
\begin{equation}\label{th2-2D}
\{ \hat A_{\boldsymbol{\alpha},\mathbf{m}}  \}_\mathbf{m}  \sim_{\lambda}  - \epsilon^2   f_{\alpha_{\hat j}}(x_{\hat j}).
\end{equation}
\end{theorem}
\begin{proof}
When
$\lim_{h_t,h_x \rightarrow 0^+} \gamma(h_t,h_x)=\ \gamma^*<\infty$, we can write
\[
A_{\alpha,M}=I_M-\epsilon^2 \gamma^*G_{\alpha,M} + sh_t I_M - (\gamma(h_t,h_x)-\gamma^*)\epsilon^2 G_{\alpha,M}
\]
with $G_{\alpha,M}$ as in (\ref{eq:BM})-(\ref{eq:Toeplitz matrix}). Hence, taking into account Proposition \ref{Proposition2}, $A_{\alpha,M}$ coincides with the Toeplitz matrix $I_M-\epsilon^2 \gamma^*G_{\alpha,M}=T_M(1 - \epsilon^2 \gamma^* f_\alpha)$ plus a correction of infinitesimal norm
$sh_t I_M - (\gamma(h_t,h_x)-\gamma^*)\epsilon^2 T_M(f_\alpha)$, as both $h_x$ and $h_t$ tend to zero. More precisely, we obtain
\begin{equation}\label{th1-2D}
A_{\alpha,M} =  T_M(1 - \epsilon^2 \gamma^* f_\alpha) + N_M, \ \ \ \|N_M\|\le \psi(h_t,h_x)\equiv sh_t + |\gamma(h_t,h_x)-\gamma^*|\epsilon^2 \|f_{\alpha}\|_\infty,
\end{equation}
since the spectral norm or induced Euclidean norm of $T_M(f)$ is always bounded by the infinity norm of the generating function $f$; see
\cite[Chapter 6]{GLT-I}.

Since $\lim_{h_t,h_x \rightarrow 0^+}\psi(h_t,h_x)=0$ every singular value of $N_M$ is infinitesimal, which implies
$\{ N_M  \}_M\sim_\sigma 0$ in the sense of Definition \ref{def5}, that is $\{ N_M  \}_M\sim_{\rm GLT} 0$, by Axiom {\bf GLT 3}, part 3 in \cite[Chapter 9, p. 170]{GLT-I}. Furthermore, by Axiom {\bf GLT 3}, part 1 in \cite[Chapter 9, p. 170]{GLT-I}, we have $\{ T_M(1 - \epsilon^2 \gamma^* f_\alpha)  \}_M\sim_{\rm GLT} 1 - \epsilon^2 \gamma^* f_\alpha$, so that the $*$ algebra structure of the GLT matrix-sequences (see Axiom {\bf GLT 4} in \cite[Chapter 9, p. 170]{GLT-I}), we deduce that
\[
\{ A_{\alpha,M}  \}_M\sim_{\rm GLT} 1 - \epsilon^2 \gamma^* f_\alpha
\]
and since all the involved matrices are Hermitian, by Axiom {\bf GLT 1} in \cite[Chapter 9, p. 170]{GLT-I}, we infer the validity of (\ref{th1-1D}).

When $\lim_{h_t,h_x \rightarrow 0^+} \gamma(h_t,h_x)=\ \infty$, we can write
\[
\hat A_{\alpha,M}=\equiv \gamma^{-1}(h_t,h_x)A_{\alpha,M} =  \gamma^{-1}(1+ sh_t)I_M-\epsilon^2 G_{\alpha,M},
\]
so that calling $N_M=\gamma^{-1}(1+ sh_t)I_M$, we deduce $\{ N_M  \}_M\sim_{\rm GLT} 0$ and
$\{ T_M(- \epsilon^2 f_\alpha)  \}_M\sim_{\rm GLT} - \epsilon^2 f_\alpha$, by Axiom {\bf GLT 3}, part 3 and part 1 in \cite[Chapter 9, p. 170]{GLT-I}.
Now, by invoking the $*$ algebra structure of the GLT matrix-sequences, we infer
\[
\{ \hat A_{\alpha,M}  \}_M\sim_{\rm GLT} - \epsilon^2  f_\alpha
\]
and since all the involved matrices are again Hermitian, by Axiom {\bf GLT 1} in \cite[Chapter 9, p. 170]{GLT-I}, we conclude the validity of (\ref{th2-1D}).

Now we consider the twolevel setting and we show that the essentials of the proof are already contained in the unilevel case, with the only warning of the proper notation.

If
\[
\lim_{h_t,h_{x_j} \rightarrow 0^+} \gamma_j(h_t,h_x)=\ \gamma_j^* < \infty,\ j=1,2
\]
then necessarily $\gamma_2^*=0$ if $\alpha_1>\alpha_2$, and alternatively $\gamma_1^*=0$ if $\alpha_2>\alpha_1$.

Taking into account the expression of the matrix $T_{\boldsymbol{\alpha},\mathbf{m}}$ and relations (\ref{eq:BTTB_matrix})-(\ref{eq:Final_Linearsystem_2d}), we can write
\[
A_{\boldsymbol{\alpha},\mathbf{m}}=I_M-\epsilon^2 \left(\gamma_1^* T_{\alpha_1,M_{1}} \otimes I_{M_2} + \gamma_2^* I_{M_1} \otimes T_{\alpha_2,M_{2}}\right)
 + E_{\boldsymbol{\alpha},\mathbf{m}}
\]
with $M=M(\mathbf{m})=M_1M_2$ and $N_{\mathbf{m}}$ being an error term, whose expression is reported below:
\[
N_{\mathbf{m}}= sh_t I_M - (\gamma_1(h_t,h_{x_1})-\gamma_1^*)\epsilon^2 T_{\alpha_1,M_{1}} \otimes I_{M_2}
- (\gamma_2(h_t,h_{x_2})-\gamma_2^*)\epsilon^2 I_{M_1} \otimes T_{\alpha_2,M_{2}}.
\]
Since $\|T_M(f)\|\le \|f\|_\infty$ (see \cite[Chapter 3]{GLT-II}) and since $\|X\otimes Y\|=\|X\|\|Y\|$, we infer
\[
\|N_{\mathbf{m}}\| \le \psi(h_t,h_{x_1},h_{x_2})\equiv sh_t + |\gamma_1(h_t,h_{x_1})-\gamma_1^*|\epsilon^2 \|f_{\alpha_1}\|_\infty
+ |\gamma_1(h_t,h_{x_2})-\gamma_2^*|\epsilon^2 \|f_{\alpha_2}\|_\infty,
\]
so that
 $A_{\boldsymbol{\alpha},\mathbf{m}}$ coincides with a twolevel Toeplitz matrix plus a correction of infinitesimal norm, as $h_{x_1},h_{x_2},h_t$ tend to zero. As a consequence,
$\{ N_{\mathbf{m}}  \}_M\sim_\sigma 0$ that is $\{ N_{\mathbf{m}}  \}_M\sim_{\rm GLT} 0$, by Axiom {\bf GLT 3}, part 3 in \cite[Chapter 6, p. 118]{GLT-II}. Furthermore,
setting $T_{\mathbf{m}}=I_M-\epsilon^2 \left(\gamma_1^* T_{\alpha_1,M_{1}} \otimes I_{M_2} + \gamma_2^* I_{M_1} \otimes T_{\alpha_2,M_{2}}\right)$, by Axiom 5 in \cite[p. 2338]{SIMAX25-Rosita} and Axiom {\bf GLT 3}, part 1 in \cite[Chapter 9, p. 170]{GLT-I}, we deduce
$\{ T_{\mathbf{m}}  \}_M\sim_{\rm GLT} 1 - \epsilon^2 ( \gamma_1^* f_{\alpha_1}(x_1)+ \gamma_2^* f_{\alpha_2}(x_2))$; see also  Axiom {\bf GLT 3}, part 1 in \cite[Chapter 6, p. 118]{GLT-II}.

Finally be the $*$ algebra structure of the GLT matrix-sequences (see Axiom {\bf GLT 4} in \cite[Chapter 6, p. 118]{GLT-II}), we deduce that
\[
A_{\boldsymbol{\alpha},\mathbf{m}}  \}_M\sim_{\rm GLT} 1 - \epsilon^2 ( \gamma_1^* f_{\alpha_1}(x_1)+ \gamma_2^* f_{\alpha_2}(x_2))
\]
and since all the involved matrices are Hermitian, by Axiom {\bf GLT 1} in \cite[Chapter 6, p. 118]{GLT-II}, we infer the validity of (\ref{th1-2D}).

When $\gamma=\max_j\gamma_j=\gamma_{\hat j}$ is unbounded with limit $\gamma^*=\infty$, taking into account proper notational adaptations, the very same steps as in the unilevel case leads to (\ref{th2-2D}).
\end{proof}


\section{$\tau$ based preconditioner}\label{sec:tau_preconditioner}

The spectral analysis in the previous section is used for suggesting proper preconditioners. In fact, from the approximation theory point of view, we need that the preconditioning matrix-sequences are such that $\{ A_{\alpha,M} - P_{\alpha,M}  \}_M$ in the unilevel setting and $\{A_{\boldsymbol{\alpha},\mathbf{m}} - P_{\boldsymbol{\alpha},\mathbf{m}}  \}_\mathbf{m}$ in the twolevel setting are zero-distributed matrix-sequences.
According to classical works \cite{linear}, we know that the best matrix algebra is the $\tau$ algebra, because our coefficient matrices are not only of Toeplitz nature, but also real symmetric: for such type of structures in fact the $\tau$ approximation choice is better than classical circulant or $\omega$ circulant approximations and this is true also in the fractional setting \cite{DMS-2016,SIMAX25-Rosita}. Following \cite{EstaS-LAA,SIMAX25-Rosita}, for $f$ H\"older continuous with any positive parameter and $2\pi$-periodic, we know that $\{ T_M(f) - \tau(T_M(f))  \}_M$
is strongly clustered at zero and it is zero-distributed under the weaker assumption that $f\in L^2(-\pi,\pi)$.
More precisely, let $T_M$ be a Toplitz matrix with the first column given by $c = [c_0, \ldots, c_{M-1}]$. Then,
\[
\tau(T_M) := T_M - H(\sigma^2(c), J_M \sigma^2(c)),
\]
with $\sigma^2(c) = [c_2, \ldots, c_{M-1},0,0]$, $J_M$ flip matrix, and corresponding $H$ Hankel matrix.
Thanks to \cite{bini1983spectral}, the $\tau$ matrix can be diagonalized as
\begin{align}\label{eq: tau_matrix_diagonalization}
    \tau(\mathcal{\bf G}_{M})=S_{M_{i}}D_{M_{i}}S_{M_{i}},
\end{align}
where 
$S_{M}$ is the discrete sine transform matrix whose entries are
\begin{equation}\label{eq: sine_transform_matrix}
    S_{M}=\left[\sqrt{\frac{2}{M+1}}\sin{\bigg(\frac{ij\pi}{M+1}\bigg)}\right]_{i,j=1}^{M}.
\end{equation}

Hence, according to the discussion above, we consider the following $\tau$-based preconditioner
\begin{equation}
P_{\alpha,M} = (1 + s h_t ) I_M - \epsilon^2 h_t \overline{c}_\alpha \tau(G_{\alpha,M}),
\end{equation}
where, as already indicated, $\tau(G_{\alpha,M})$ denotes the $\tau$ matrix related to the Toeplitz matrix $G_{\alpha,M}$.

In the two-dimensional setting, the $\tau$ preconditioner  is defined accordingly, where
\begin{equation}
\tau(G_{\boldsymbol{\alpha},\mathbf{m}}) = \tau(G_{\alpha_1,M_1}) \otimes I_{M_2} + I_{M_2} \otimes \tau(G_{\alpha_2,M_2}).
\end{equation}
Clearly, $\tau(G_{\boldsymbol{\alpha},\mathbf{m}})$
can be diagonalised as $\tau(G_{\boldsymbol{\alpha},\mathbf{m}}) = S_\mathbf{m} \Lambda_\mathbf{m} S_\mathbf{m}$,
 with $S_\mathbf{m} = S_{M_1} \otimes S_{M_2}$ and $\Lambda_\mathbf{m} = \Lambda_{M_1} \otimes I_{M_2} + I_{M_1} \otimes \Lambda_{M_2}$.

%
%
\begin{theorem}\label{Proposition3_2D}
Under the same notations as in Theorem \ref{Proposition3_2D_1}, the error matrix-sequences are clustered at zero
i.e.
\begin{equation*}
\{ A_{\alpha,M} - P_{\alpha,M}  \}_M,\ \{ \hat A_{\alpha,M} - \hat P_{\alpha,M}  \}_M   \sim_{\lambda}  0
\end{equation*}
and
\begin{equation*}
\{A_{\boldsymbol{\alpha},\mathbf{m}} - P_{\boldsymbol{\alpha},\mathbf{m}}  \}_\mathbf{m}, \
\{\hat A_{\boldsymbol{\alpha},\mathbf{m}} - P_{\boldsymbol{\alpha},\mathbf{m}}  \}_\mathbf{m}  \sim_{\lambda}  0.
\end{equation*}
Finally, the corresponding preconditioned matrix-sequences both in one and two dimensions and both in the scaled and nonscaled versions are
all clustered at $1$ in weak sense.
\end{theorem}
\begin{proof}
The proof is a combination of Theorem \ref{Proposition3_2D_1} and of results in the literature. More precisely, by the proof of Theorem \ref{Proposition3_2D_1} and by the results on the $\tau$ approximation in \cite{SIMAX25-Rosita}, we know that 
\begin{equation*}
\{ A_{\alpha,M} - P_{\alpha,M}  \}_M,\ \{ \hat A_{\alpha,M} - \hat P_{\alpha,M}  \}_M   \sim_{\sigma,\rm GLT}  0
\end{equation*}
and
\begin{equation*}
\{A_{\boldsymbol{\alpha},\mathbf{m}} - P_{\boldsymbol{\alpha},\mathbf{m}}  \}_\mathbf{m}, \
\{\hat A_{\boldsymbol{\alpha},\mathbf{m}} - P_{\boldsymbol{\alpha},\mathbf{m}}  \}_\mathbf{m}  \sim_{\sigma,\rm GLT}  0.
\end{equation*}
Since all the involved matrices are Hermitian, by Axiom {\bf GLT 1} in \cite[Chapter 9, p. 170]{GLT-I} and by Axiom {\bf GLT 1} in \cite[Chapter 6, p. 118]{GLT-II}, we infer 
\begin{equation*}
\{ A_{\alpha,M} - P_{\alpha,M}  \}_M,\ \{ \hat A_{\alpha,M} - \hat P_{\alpha,M}  \}_M   \sim_{\lambda}  0
\end{equation*}
and
\begin{equation*}
\{A_{\boldsymbol{\alpha},\mathbf{m}} - P_{\boldsymbol{\alpha},\mathbf{m}}  \}_\mathbf{m}, \
\{\hat A_{\boldsymbol{\alpha},\mathbf{m}} - P_{\boldsymbol{\alpha},\mathbf{m}}  \}_\mathbf{m}  \sim_{\lambda}  0.
\end{equation*}
Now the part regarding the preconditioning, the claimed thesis can be obtained verbatim following \cite[Exercize 8.4, pp. 162,291,292]{GLT-I} for the unilevel case and \cite[Theorem 7.1, p. 122]{GLT-II} for the twolevel case.
\end{proof}

\section{Numerical Section}\label{sec:num}
In this section, we first present numerical examples to verify the effectiveness of the $\tau$-based preconditioner defined in the previous section when is combined with Conjugate Gradient (CG). Computations are performed using the built-in \textit{pcg} MATLAB function.

 Regarding the computational point of view, $P_{\alpha,M}$ is extremely suitable because the matrix-vector product can be performed in $O(M\log M)$ operations through the discrete sine transform (DST) algorithm. More precisely, the constant in the big $O$ can be precisely estimated and, following \cite{VanLoan}, the DST algorithm in terms of real operations costs around one half of the fast Fourier transform.

The stopping criterion is chosen as
\begin{equation*}
    \frac{\| r^{k}\|}{\| r^{0}\|}<10^{-8},
\end{equation*}
where $\|\cdot\|$ denotes the Euclidean norm, $r^{k}$ is the residual vector at the $k$-th iteration. The initial guess is fixed as the zero vector. All the numerical experiments are run by using MATLAB R2021a on HP 17-cp0000nl computer with configuration, AMD Ryzen 7 5700U with Radeon Graphics CPU and 16GB RAM.

\subsection{Preconditioning Results}
{\bf Problem 1:}\label{1Dex_Pre}
We consider the 1D-SFAC equation given in Eq. \eqref{eq:SFAC} on the domain $\Omega = \left(0, 1\right)$ with the initial condition
\begin{equation}
u(x,0)= x^{3}\left(1-x\right)^{3}.
\end{equation}
  \begin{table}[]
	\caption{$d=1,~N=2^{4},~\epsilon=0.3$, the average number of iterations of PCG for different values of $\alpha$.}
	\begin{small}
		\setlength{\tabcolsep}{9.4pt}
		\begin{center}
			\begin{tabular}{c c c c c c c c c c}
			\hline
			\hline
     \noalign{\vskip 1mm}\multirow{1}{*} {}&
				\multicolumn{1}{c}{${}$} &
                \multicolumn{2}{c}{$(p,q)=(1,0)$} &
                \multicolumn{2}{c}{$(p,q)=(1,-1)$} \\
                \cmidrule(r){3-4}\cmidrule(r){5-6}
				${}$&$M+1$& ${\mbox{CG}}$& ${P_{\alpha,M}}$&${\mbox{CG}}$& ${P_{\alpha,M}}$\\ \hline
    \noalign{\vskip 1mm}
          ${}$&${2^6   }$&${6  }$&${4}$&${10 }$&${4}$\\
          ${}$&${2^7   }$&${8  }$&${4}$&${13 }$&${5}$\\
          ${}$&${2^8   }$&${10 }$&${5}$&${19 }$&${5}$\\
${\alpha=1.2}$&${2^9   }$&${13 }$&${5}$&${28 }$&${5}$\\
          ${}$&${2^{10}}$&${19 }$&${5}$&${41 }$&${5}$\\
          ${}$&${2^{11}}$&${28 }$&${5}$&${63 }$&${6}$\\
          ${}$&${2^{12}}$&${43 }$&${6}$&${97 }$&${6}$\\
          ${}$&${2^{13}}$&${65 }$&${6}$&${149}$&${6}$\\
\hline
\noalign{\vskip 1mm}
          ${}$&${2^6   }$&${9  }$&${4}$&${12 }$&${4}$\\
          ${}$&${2^7   }$&${13 }$&${4}$&${18 }$&${5}$\\
          ${}$&${2^8   }$&${21 }$&${5}$&${29 }$&${5}$\\
${\alpha=1.5}$&${2^9   }$&${35 }$&${5}$&${49 }$&${5}$\\
          ${}$&${2^{10}}$&${58 }$&${5}$&${82 }$&${6}$\\
          ${}$&${2^{11}}$&${99 }$&${6}$&${140}$&${6}$\\
          ${}$&${2^{12}}$&${167}$&${6}$&${240}$&${6}$\\
          ${}$&${2^{13}}$&${287}$&${6}$&${411}$&${6}$\\
\hline
\noalign{\vskip 1mm}
          ${}$&${2^6   }$&${14  }$&${4}$&${15  }$&${4}$\\
          ${}$&${2^7   }$&${24  }$&${4}$&${27  }$&${4}$\\
          ${}$&${2^8   }$&${44  }$&${5}$&${49  }$&${5}$\\
${\alpha=1.8}$&${2^9   }$&${82  }$&${5}$&${91  }$&${5}$\\
          ${}$&${2^{10}}$&${155 }$&${5}$&${173 }$&${5}$\\
          ${}$&${2^{11}}$&${294 }$&${5}$&${328 }$&${5}$\\
          ${}$&${2^{12}}$&${558 }$&${6}$&${624 }$&${5}$\\
          ${}$&${2^{13}}$&${1062}$&${6}$&${1192}$&${6}$\\
\hline
			\end{tabular}
		\end{center}
	\end{small}
	\label{tab:T11}
\end{table}
\begin{figure}
		\centering
		\begin{subfloat}[$(p,q)=(1,0).$]
			{\resizebox*{7cm}{!}{\includegraphics[width=\textwidth]{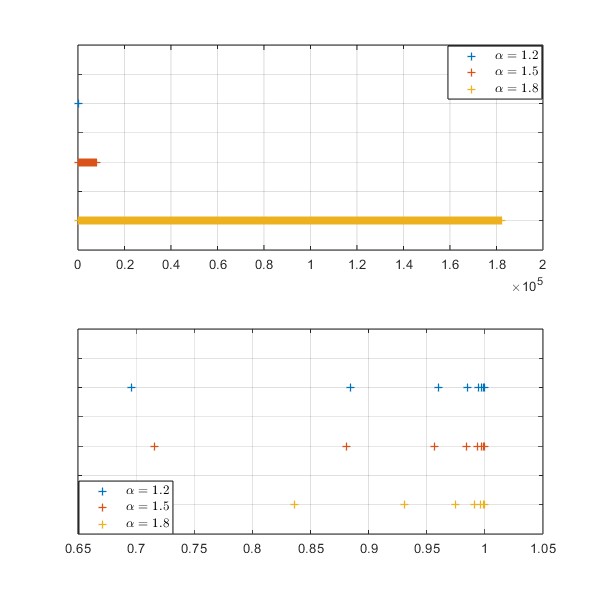}}}
		\end{subfloat}
		\begin{subfloat}[$(p,q)=(1,-1).$]
			{\resizebox*{7cm}{!}{\includegraphics[width=\textwidth]{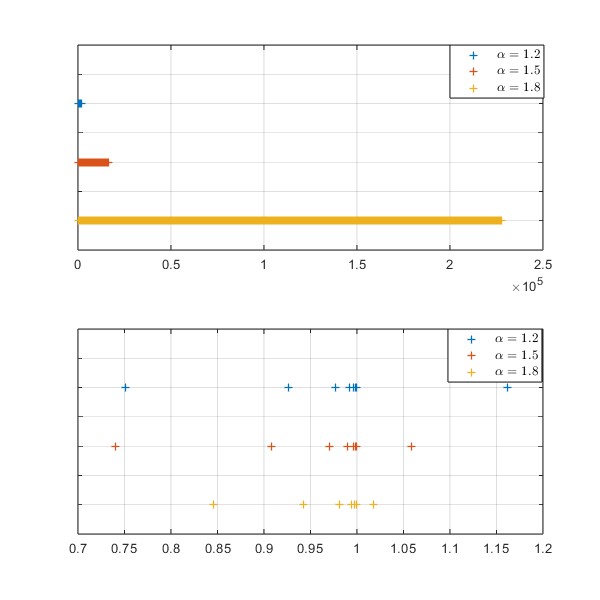}}}
		\end{subfloat}
		\caption{Spectral distributions of unpreconditioned coefficient matrices $A_{\alpha,M}$ and preconditioned $P_{\alpha,M}^{-1}A_{\alpha,M}$ of Example 2 for different fractional derivatious  $\alpha_{i}$ fixing $N=2^{4},~M=2^{13}$.}
	\label{fig: Spectral_analysis_1D}
	\end{figure}
{\bf Problem 2:}\label{2Dex1_Pre}
Consider the 2D-SFACH equation defined on $\Omega = \left(0, 1\right)^{2}$,
subject to the initial condition
\begin{equation}
u(x,y,0)=0.05\sin{(2\pi x)}\sin{(2\pi y)}.
\end{equation}
To demonstrate the efficiency of the proposed preconditioning technique, the PCG is applied as an iterative solver to solve both the 1D and 2D problems, with $T = 1,~N=2^{4}$ and $M_{1} = M_{2}$. The spectral distributions of the original coefficient matrix and the preconditioned one are shown in
Figs.~\ref{fig: Spectral_analysis_1D}-\ref{fig: Spectral_analysis_2D}, while the results in terms of the rounded average number of iterations are listed in Tables~\ref{tab:T11}-\ref{tab:T21}. As it can be observed, the preconditioning technique is not only optimal in terms of computational cost per iteration, but also the number of iterations is quite stable and is again optimal since tends to a constant independent of the space and time fineness parameter: this result is better than expected because the matrix structures are multilevel and because of the theoretical barriers proven in \cite{nega1,nega2}. 

  \begin{table}[]
	\caption{$d=2,~N=2^{4},~M_{1}=M_{2},~\epsilon=0.3$, the average number of iterations of PCG for different values of $\boldsymbol{\alpha}$.}
	\begin{small}
		\setlength{\tabcolsep}{9.4pt}
		\begin{center}
			\begin{tabular}{c c c c c c c c c c}
			\hline
			\hline
     \noalign{\vskip 1mm}\multirow{1}{*} {}&
				\multicolumn{1}{c}{${}$} &
                \multicolumn{2}{c}{$(p,q)=(1,0)$} &
                \multicolumn{2}{c}{$(p,q)=(1,-1)$} \\
                \cmidrule(r){3-4}\cmidrule(r){5-6}
				${(\alpha,\beta)}$&$M1+1$& ${\mbox{CG}}$& ${P_{\boldsymbol{\alpha},\mathbf{m}}}$&${\mbox{CG}}$& ${P_{\boldsymbol{\alpha},\mathbf{m}}}$\\ \hline
    \noalign{\vskip 1mm}
                    ${}$&${2^5   }$&${7  }$&${5}$&${10}$&${5}$\\
                    ${}$&${2^6   }$&${9  }$&${5}$&${14}$&${5}$\\
                    ${}$&${2^7   }$&${11 }$&${6}$&${22}$&${6}$\\
${\left(1.1,1.3\right)}$&${2^8   }$&${16 }$&${6}$&${33}$&${6}$\\
                    ${}$&${2^9   }$&${24 }$&${7}$&${51}$&${7}$\\
                    ${}$&${2^{10}}$&${37 }$&${7}$&${79}$&${8}$\\

\hline
\noalign{\vskip 1mm}
                    ${}$&${2^5   }$&${10 }$&${5}$&${11 }$&${5}$\\
                    ${}$&${2^6   }$&${15 }$&${5}$&${19 }$&${5}$\\
                    ${}$&${2^7   }$&${26 }$&${6}$&${33 }$&${6}$\\
${\left(1.5,1.6\right)}$&${2^8   }$&${43 }$&${6}$&${57 }$&${6}$\\
                    ${}$&${2^9   }$&${74 }$&${7}$&${100}$&${6}$\\
                    ${}$&${2^{10}}$&${130}$&${7}$&${175}$&${7}$\\
\hline
\noalign{\vskip 1mm}
                    ${}$&${2^5   }$&${14 }$&${5}$&${15 }$&${5}$\\
                    ${}$&${2^6   }$&${26 }$&${5}$&${28 }$&${5}$\\
                    ${}$&${2^7   }$&${48 }$&${6}$&${52 }$&${6}$\\
${\left(1.7,1.9\right)}$&${2^8   }$&${93 }$&${6}$&${101}$&${6}$\\
                    ${}$&${2^9   }$&${181}$&${6}$&${196}$&${6}$\\
                    ${}$&${2^{10}}$&${355}$&${7}$&${383}$&${7}$\\
\hline
\noalign{\vskip 1mm}
                    ${}$&${2^5   }$&${14 }$&${4}$&${15 }$&${5}$\\
                    ${}$&${2^6   }$&${25 }$&${5}$&${27 }$&${5}$\\
                    ${}$&${2^7   }$&${48 }$&${5}$&${52 }$&${5}$\\
${\left(1.1,1.9\right)}$&${2^8   }$&${95 }$&${6}$&${100}$&${6}$\\
                    ${}$&${2^9   }$&${188}$&${6}$&${201}$&${6}$\\                ${}$&${2^{10}}$&${375}$&${7}$&${399}$&${6}$\\
\hline

            \end{tabular}
		\end{center}
	\end{small}
	\label{tab:T21}
\end{table}
\begin{figure}
		\centering
		\begin{subfloat}[$(p,q)=(1,0).$]
			{\resizebox*{7cm}{!}{\includegraphics[width=\textwidth]{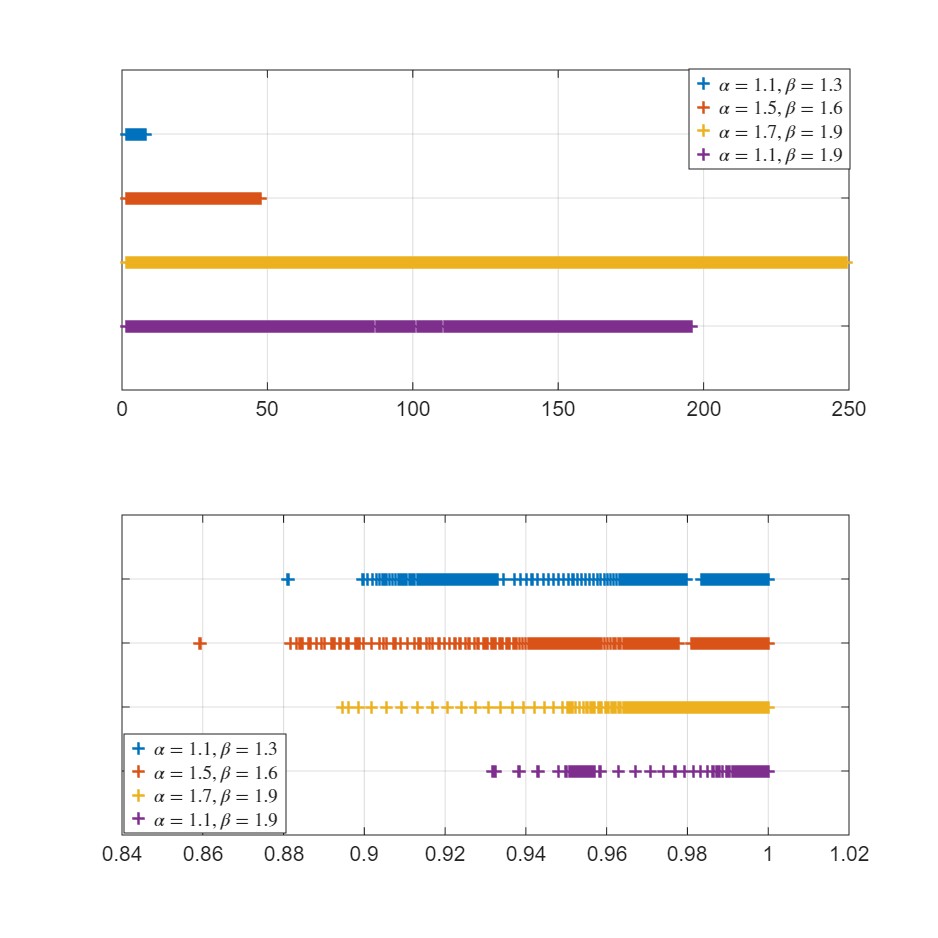}}}
		\end{subfloat}
		\begin{subfloat}[$(p,q)=(1,-1).$]
			{\resizebox*{7cm}{!}{\includegraphics[width=\textwidth]{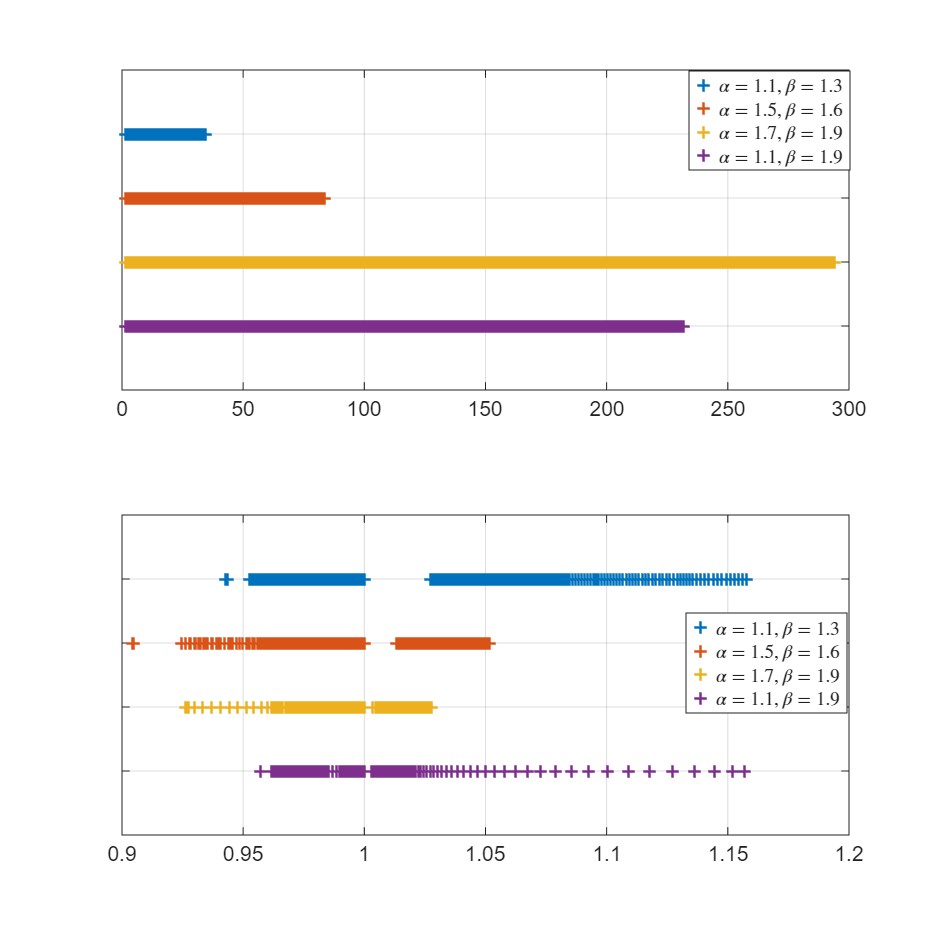}}}
		\end{subfloat}
		\caption{Spectral distributions of unpreconditioned coefficient matrices $A_{\boldsymbol{\alpha},\mathbf{m}}$ and preconditioned ones $P_{\boldsymbol{\alpha},\mathbf{m}}^{-1}A_{\boldsymbol{\alpha},\mathbf{m}}$ of Example 2 for different $\boldsymbol{\alpha}$ values, fixing $N=2^{4},~M_{1}=M_{2}=2^{7}$.}
	\label{fig: Spectral_analysis_2D}
	\end{figure}
In the following, we discuss the numerical results in both 1D and 2D obtained using the proposed stabilized scheme \cite{zhang2025efficient}. The effect of the fractional order on the interface evolution is studied extensively. For all the problems, we use zero Dirichlet boundary conditions.

{\bf Problem 3:}\label{ex1}
In this example we study the effect of fractional order on the interface profile. We consider the 1D-SFAC equation given in Eq. \eqref{eq:SFAC} with the initial condition
\begin{equation}
u(x,0)= 0.5\sin{2\pi x}, \quad \quad \Omega = \left(0, 1\right).
\end{equation}
The grid points along the spatial axis are $M = 256$ and the final simulation time is taken as $T = 12$ with $N=128$ grid points along the temporal direction. The other parameter values are $\epsilon=0.02$. The numerical results obtained for different fractional orders $\alpha$ are given in Fig. \ref{fig1:main}. It is evident from Fig. \ref{subfig:fig1a} and \ref{subfig:fig1b} that the interface profile become very sharp when $\alpha$ is smaller and more diffusive when it is larger as depicted from Fig. \ref{subfig:fig1c} and \ref{subfig:fig1d}. The energy profile for different fractional order is given in Fig. \ref{subfig:fig2a}. It is observed that the energy profile is dissipative and behaves almost the same for all fractional order. The maximum value of the solution profile over time for all fractional orders is given in Fig. \ref{subfig:fig2b} which demonstrate that the maximum value is bounded by $0.9575$ which is in accordance with the theoretical results given in \cite{zhang2025efficient}.\\
\begin{figure}[ht!]
\centering
\subfloat[\label{subfig:fig1a}]{\includegraphics[width=0.5\linewidth]{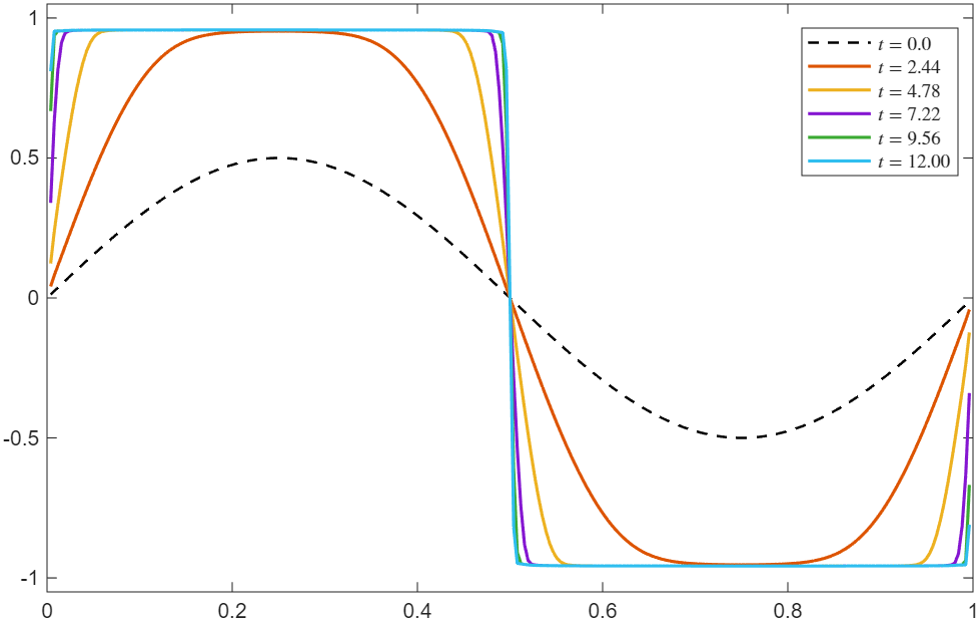}}\hfil
\subfloat[\label{subfig:fig1b}]{\includegraphics[width=0.5\linewidth]{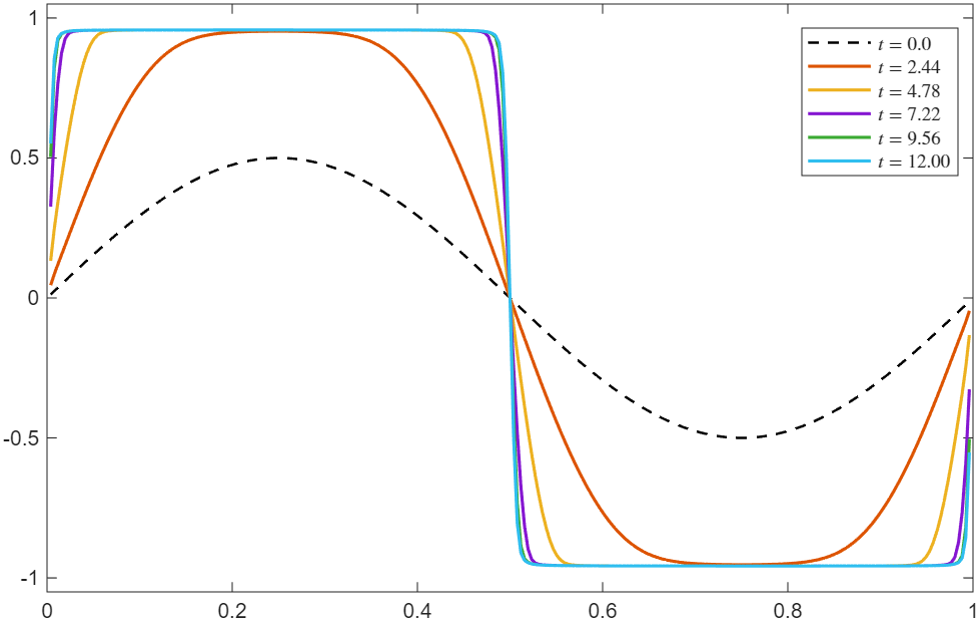}}\par
\subfloat[\label{subfig:fig1c}]{\includegraphics[width=0.5\linewidth]{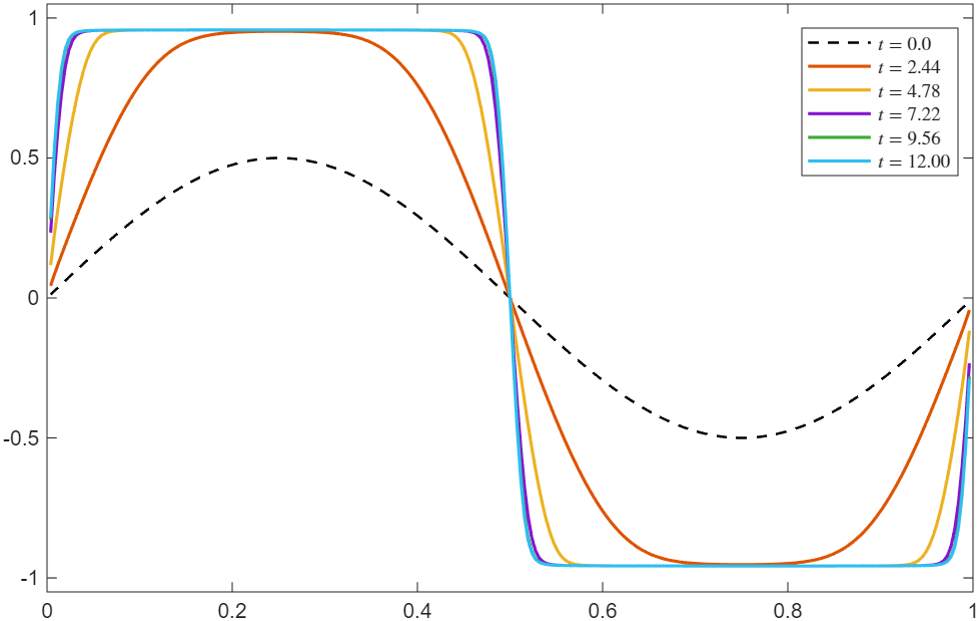}}\hfil
\subfloat[\label{subfig:fig1d}]{\includegraphics[width=0.5\linewidth]{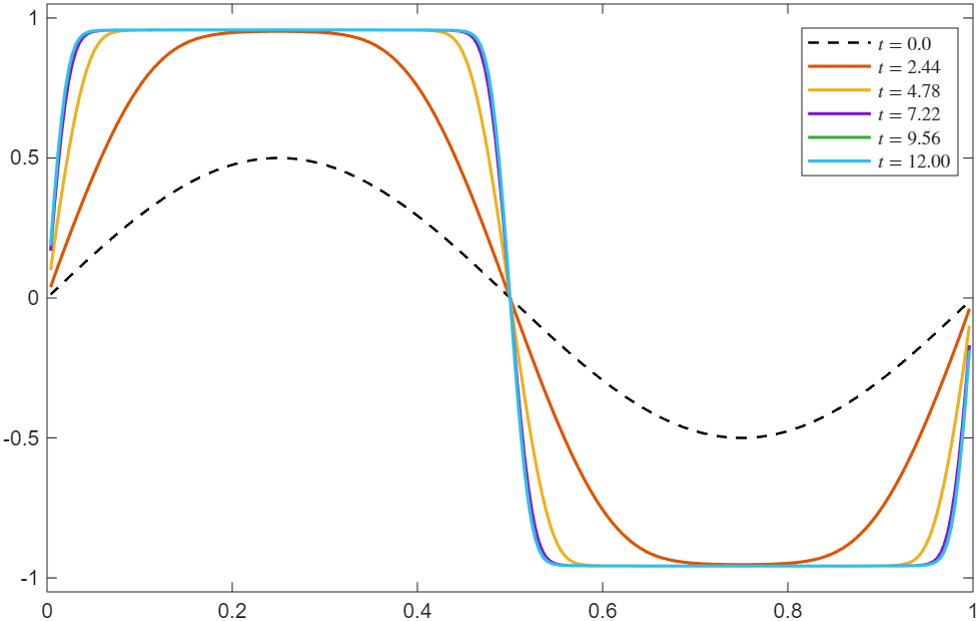}}
\caption{Solution profile for problem 3 when (a) $\alpha$= 1.25 (b) $\alpha$= 1.5 (c) $\alpha$= 1.75 (d) $\alpha$= 1.9.}
\label{fig1:main}
\end{figure}
\begin{figure}[ht!]
\centering
\subfloat[\label{subfig:fig2a}]{\includegraphics[width=0.5\linewidth]{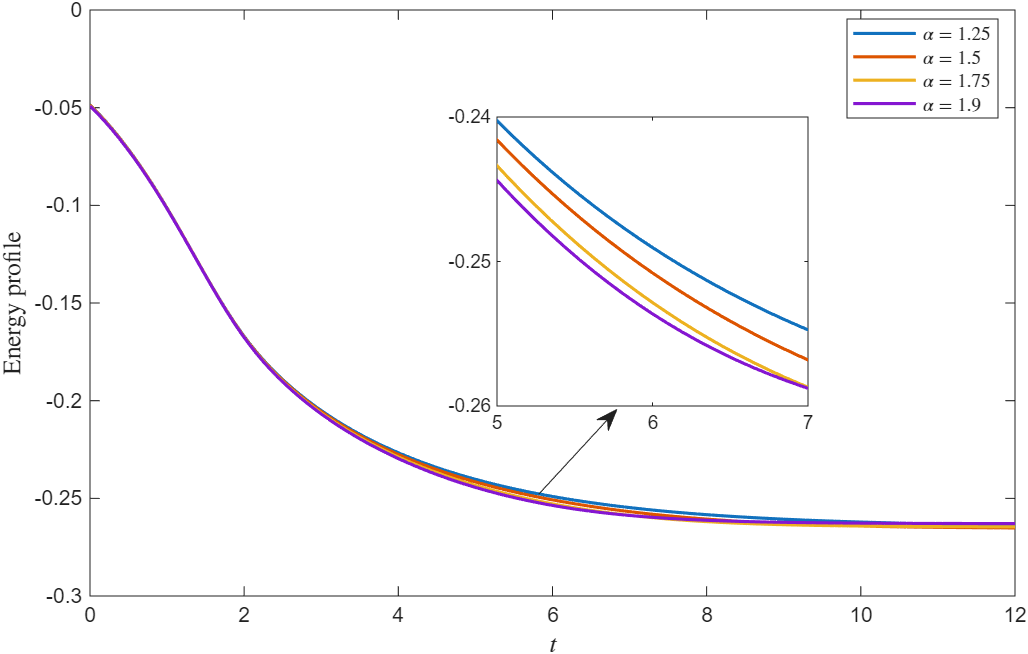}}\hfil
\subfloat[\label{subfig:fig2b}]{\includegraphics[width=0.5\linewidth]{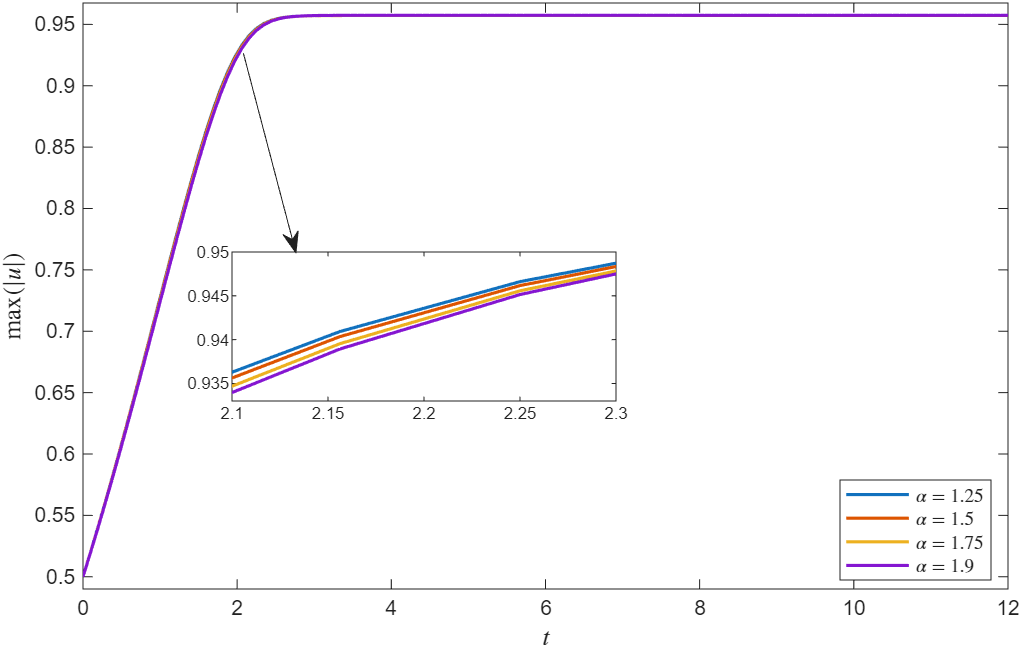}}
\caption{(a) Energy profile for problem 3 (b) Maximum value of the solution with respect to time for problem 3.}
\label{fig2:main}
\end{figure}
{\bf Problem 4:}\label{ex2}
In this problem, we consider the dynamics of the star shape given by the following initial condition
\begin{equation*}
u(x,y,0)=
\begin{cases}
\;\;0.99,
& \displaystyle -\tanh\!\left(\frac{C(x,y)}{E}\right) > 0.99, \\[10pt]
-0.99,
& \displaystyle -\tanh\!\left(\frac{C(x,y)}{E}\right) < -0.99, \\[10pt]
\displaystyle -\tanh\!\left(\frac{C(x,y)}{E}\right),
& \text{otherwise},
\end{cases}
\end{equation*}
where
\[
C(x,y)
=
\min\!\Big\{
\max\big(|2x-1|,\;|y-0.5|\big),\;
\max\big(|x-0.5|,\;|2y-1|\big)
\Big\}
-0.2,
\]
and $E=\sqrt{2}\epsilon$. The domain for this problem is a unit square. The number of grid points along both spatial directions are $M_{1} = M_{2} = 128$, while the final simulation is taken as $T = 60$ and the other parameter values are $\alpha=\beta$, $\epsilon = 0.02$ and $h_t = 0.06$. The dynamics of the interface for different values of $\alpha$ are given in Fig. \ref{fig7:main}. The first row of Fig. \ref{fig7:main} shows the results for $\alpha = 1.3$ indicating the dynamics are slow and the interface profile is very sharp. The second row of Fig. \ref{fig7:main} depicts the results for $\alpha = 1.5$ which reveal that the dynamics are now faster than that with $\alpha = 1.3$ and the interface profile is a bit diffusive. The last row of Fig. \ref{fig7:main} gives the results for $\alpha = 1.8$ which show that the interface profile is more diffusive compared to $\alpha = 1.3, 1.5$ and the dynamics are also faster as compared to other values of the fractional order.\\
\begin{figure}[ht!]
\centering
\subfloat[$t=0.06$\label{subfig:fig7a}]{\includegraphics[width=0.18\linewidth]{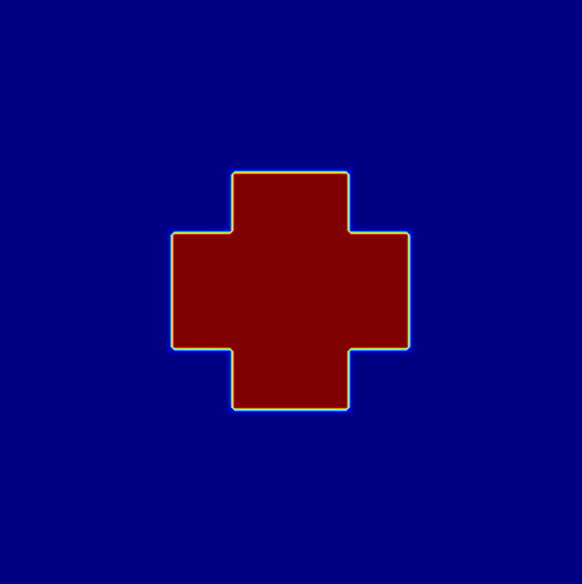}}\hfil
\subfloat[$t=19.98$\label{subfig:fig7b}]{\includegraphics[width=0.18\linewidth]{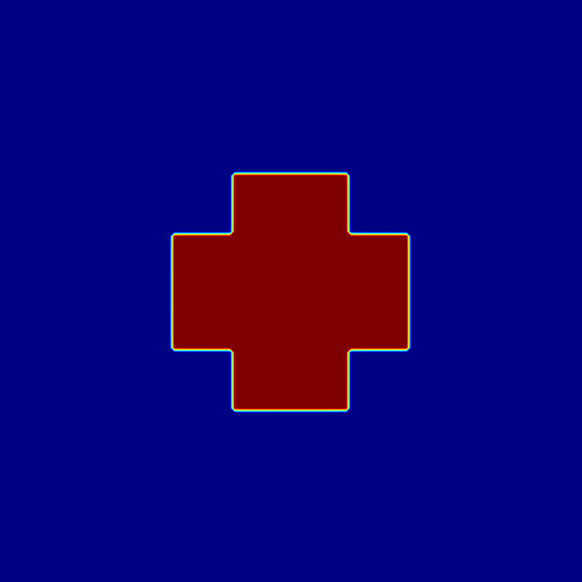}}\hfil
\subfloat[$t=40.02$\label{subfig:fig7c}]{\includegraphics[width=0.18\linewidth]{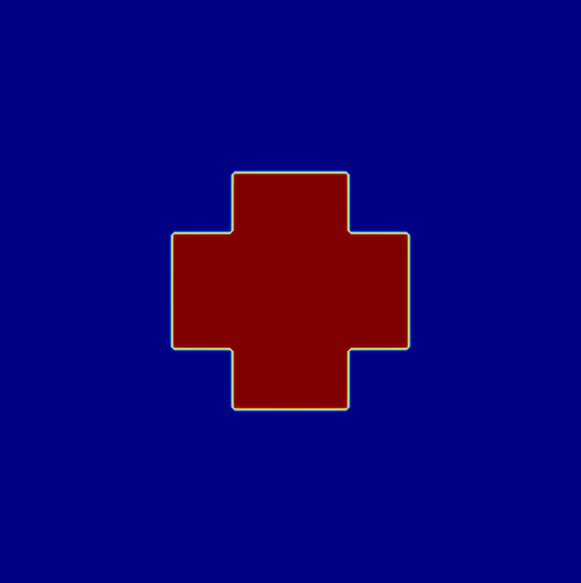}}\hfil
\subfloat[$t=60$\label{subfig:fig7d}]{\includegraphics[width=0.18\linewidth]{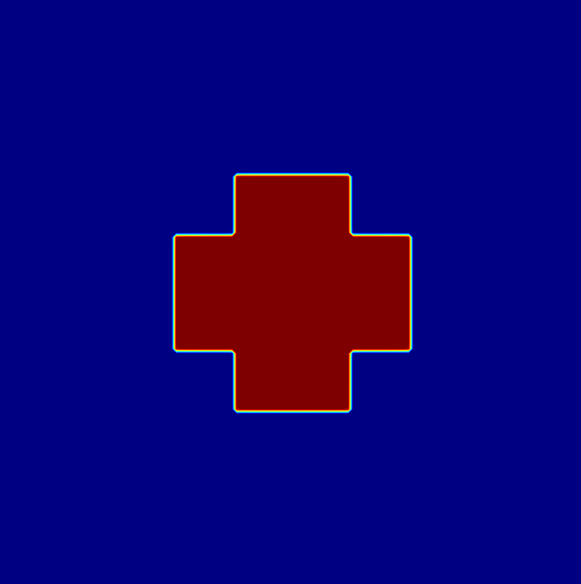}}\par
\subfloat[$t=0.06$\label{subfig:fig7e}]{\includegraphics[width=0.18\linewidth]{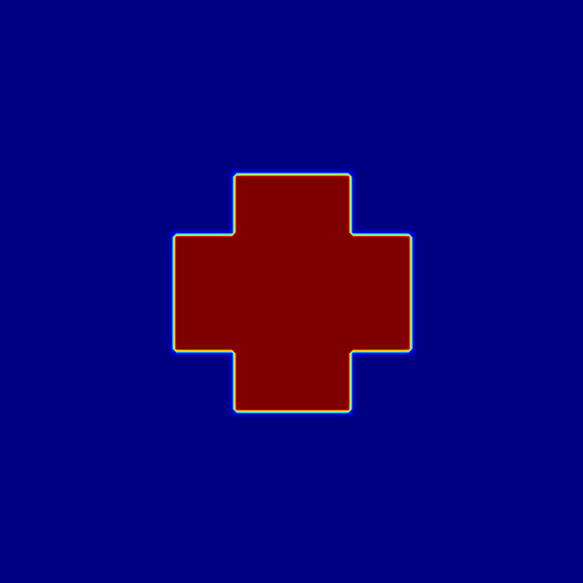}}\hfil
\subfloat[$t=19.98$\label{subfig:fig7f}]{\includegraphics[width=0.18\linewidth]{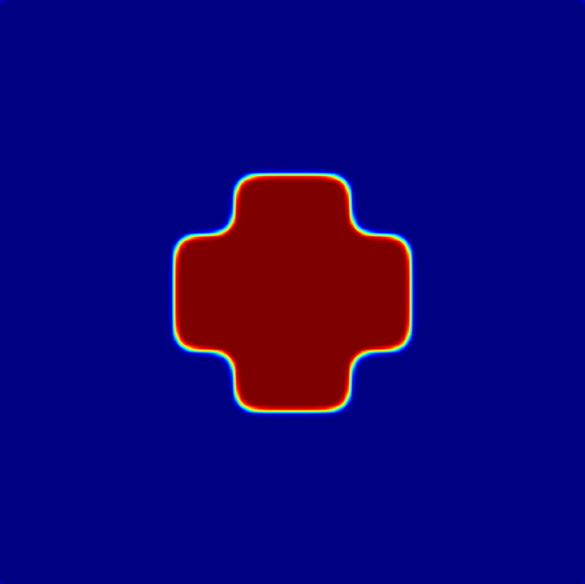}}\hfil
\subfloat[$t=40.02$\label{subfig:fig7g}]{\includegraphics[width=0.18\linewidth]{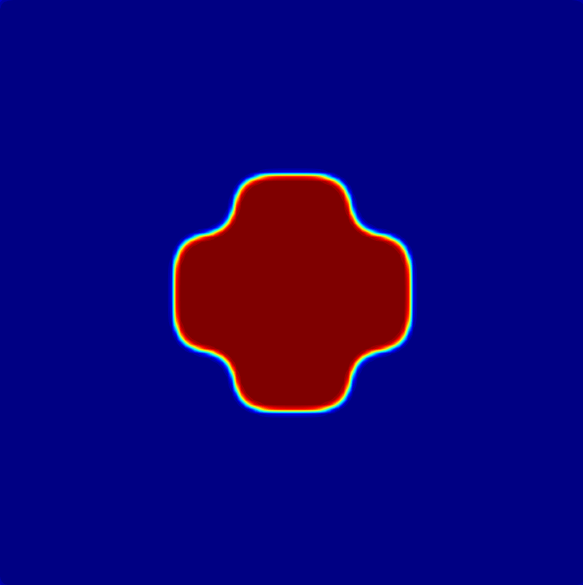}}\hfil
\subfloat[$t=60$\label{subfig:fig7h}]{\includegraphics[width=0.18\linewidth]{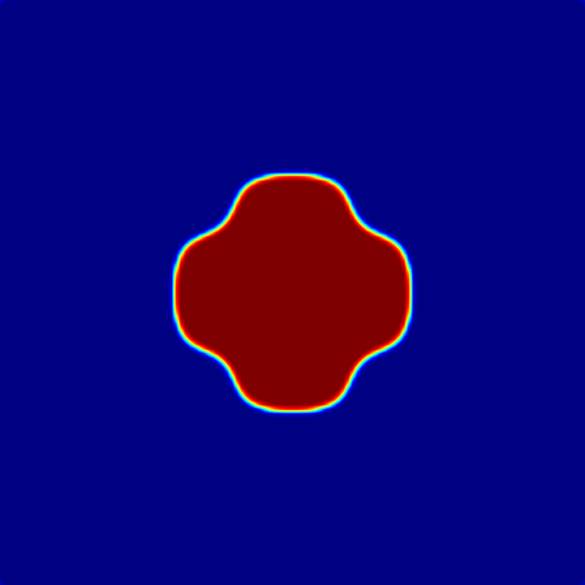}}\par
\subfloat[$t=0.06$\label{subfig:fig7i}]{\includegraphics[width=0.18\linewidth]{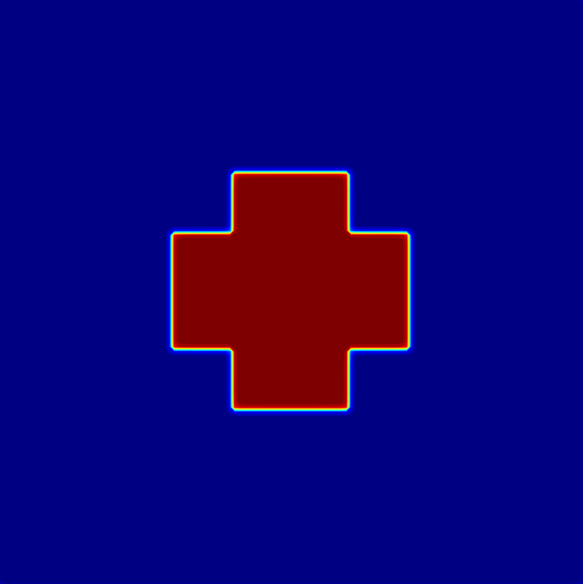}}\hfil
\subfloat[$t=19.98$\label{subfig:fig7j}]{\includegraphics[width=0.18\linewidth]{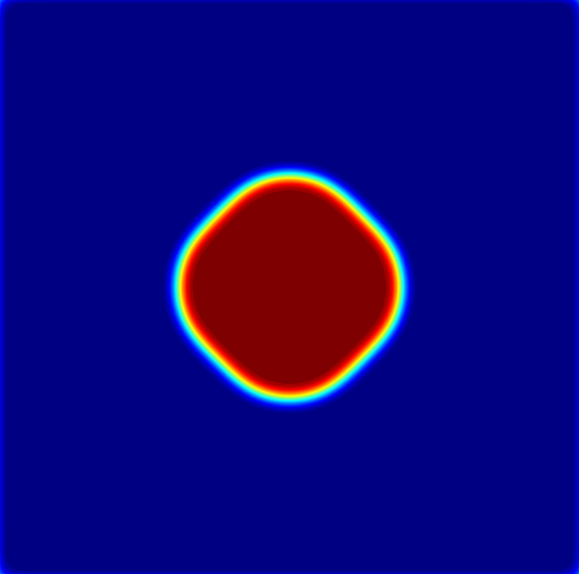}}\hfil
\subfloat[$t=40.02$\label{subfig:fig7k}]{\includegraphics[width=0.18\linewidth]{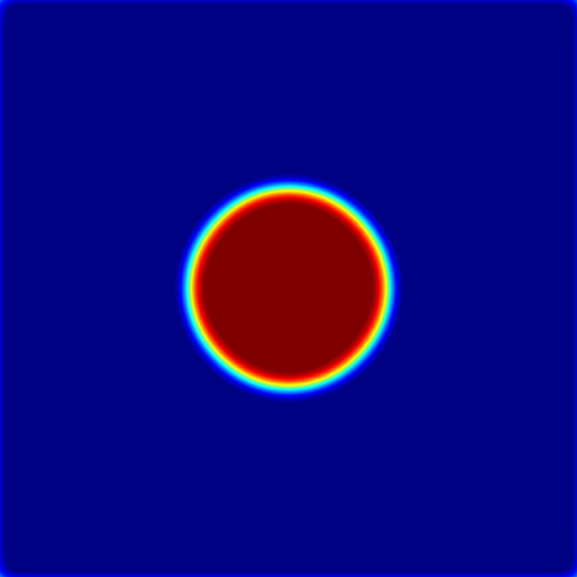}}\hfil
\subfloat[$t=60$\label{subfig:fig7l}]{\includegraphics[width=0.18\linewidth]{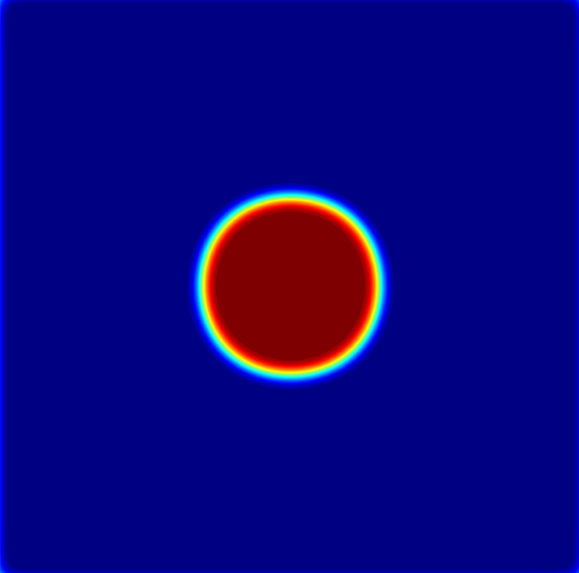}}
\caption{Dynamic of the star shape for (i) $\alpha=1.3$ (first row) (ii) $\alpha=1.5$ (second row) (iii) $\alpha=1.8$ (third row)}
\label{fig7:main}
\end{figure}
{\bf Problem 5:}
In this problem,  we have investigated the effect of fractional order on the interface profile and overall morphology evolution in a well-known phenomenon of material science known as spinodal decomposition. In this process, initially a mixture is taken, and it eventually separates into different phases, forming complex interface evolution over time. The initial condition is generated using random numbers given by
\begin{equation*}
u(x,y,0) = 0.01(2\text{rand}(x, y)-1).
\end{equation*}
The domain is taken as a unit square discretized uniformly along both spatial directions with $M_{1} = M_{2}= 128$ and the final simulation time for this problem is also $T= 40$ with $N = 1000$ grid points in the temporal direction. The interface width is $\epsilon = 0.02$, while the fractional derivatives satisfy $\alpha=\beta$. The dynamics of the interface profile for different fractional orders, i.e., $\alpha = 1.3, 1.5, 1.8$ are plotted in Fig. \ref{fig8:main}. Fig. \ref{subfig:fig8a}-\ref{subfig:fig8d} shows the morphology evolution for $\alpha = 1.3$ at different times, showing that the interface profile is very sharp with overall slow dynamics. Moreover, the structure is more dense and heterogeneous for this smaller value of the fractional order. The second row of Fig. \ref{fig8:main} depicts the results for $\alpha = 1.5$ which demonstrate that the dynamics are fast and the interface profile is diffusive as compared to $\alpha = 1.3$. Similarly, the last row of Fig. \ref{fig8:main} shows the evolution of the interface profile for $\alpha = 1.8$. It is clearly observed that the interface is more diffusive and the dynamics of the overall problem are fast. Furthermore, the structure is less heterogeneous as compared to $\alpha = 1.3, 1.5$. The numerical results plotted in Fig. \ref{fig8:main} are consistent with the results reported in \cite{sohaib2024sfac}.
\begin{figure}[ht!]
\centering
\subfloat[$t=0.06$\label{subfig:fig8a}]{\includegraphics[width=0.18\linewidth]{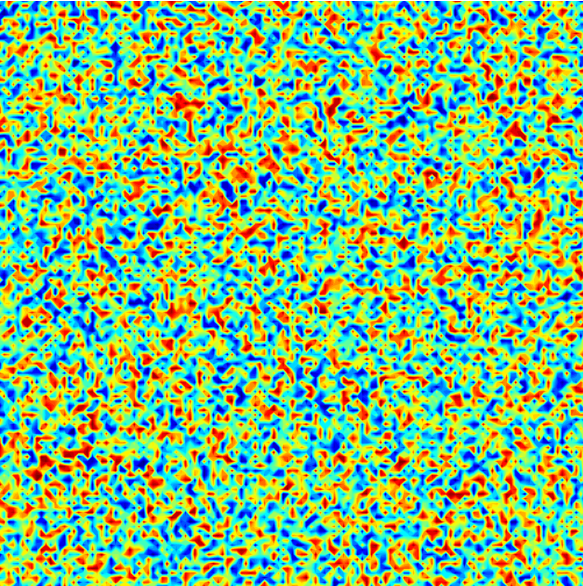}}\hfil
\subfloat[$t=19.98$\label{subfig:fig8b}]{\includegraphics[width=0.18\linewidth]{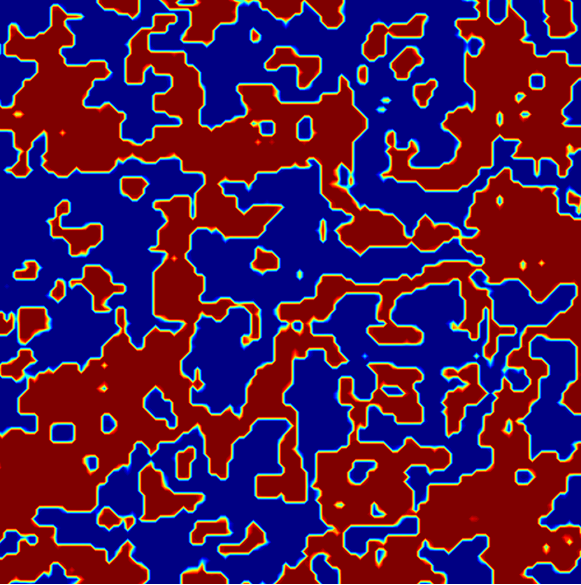}}\hfil
\subfloat[$t=40.02$\label{subfig:fig8c}]{\includegraphics[width=0.18\linewidth]{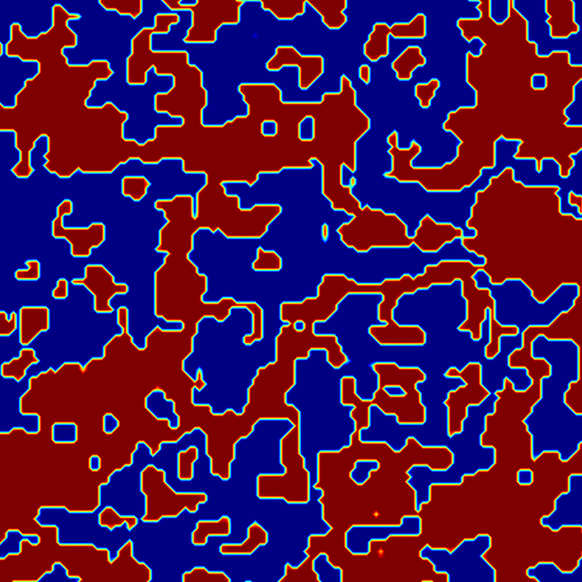}}\hfil
\subfloat[$t=60$\label{subfig:fig8d}]{\includegraphics[width=0.18\linewidth]{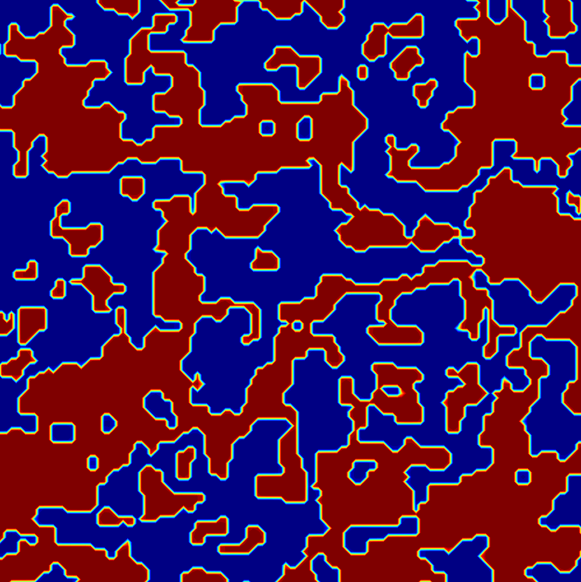}}\par
\subfloat[$t=0.06$\label{subfig:fig8e}]{\includegraphics[width=0.18\linewidth]{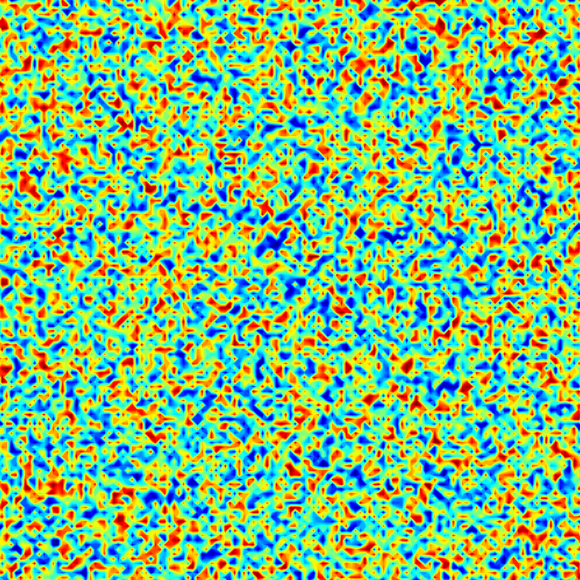}}\hfil
\subfloat[$t=19.98$\label{subfig:fig8f}]{\includegraphics[width=0.18\linewidth]{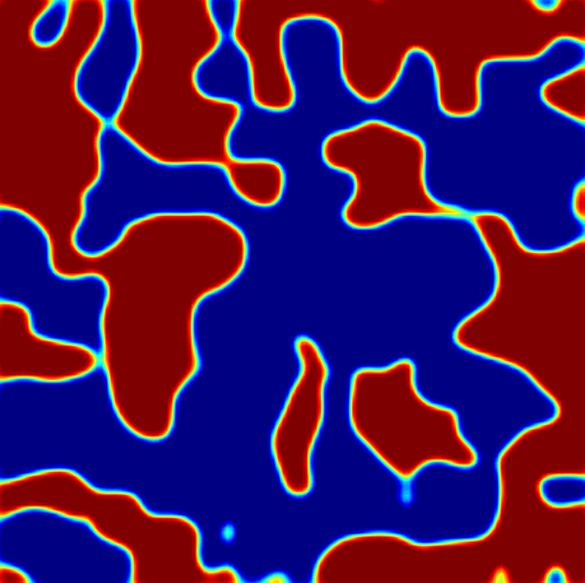}}\hfil
\subfloat[$t=40.02$\label{subfig:fig8g}]{\includegraphics[width=0.18\linewidth]{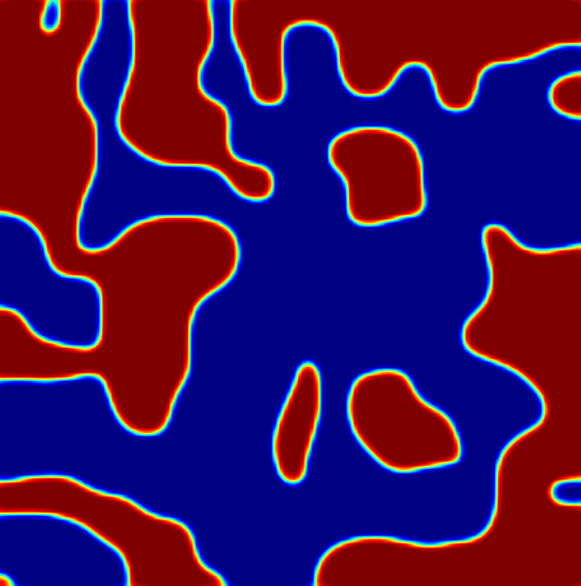}}\hfil
\subfloat[$t=60$\label{subfig:fig8h}]{\includegraphics[width=0.18\linewidth]{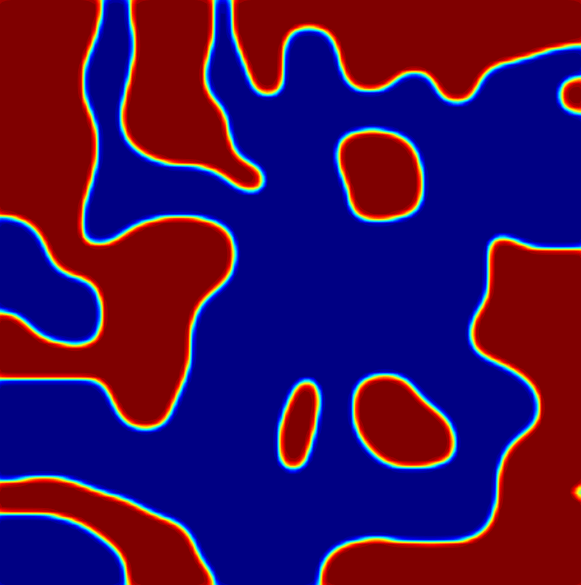}}\par
\subfloat[$t=0.06$\label{subfig:fig8i}]{\includegraphics[width=0.18\linewidth]{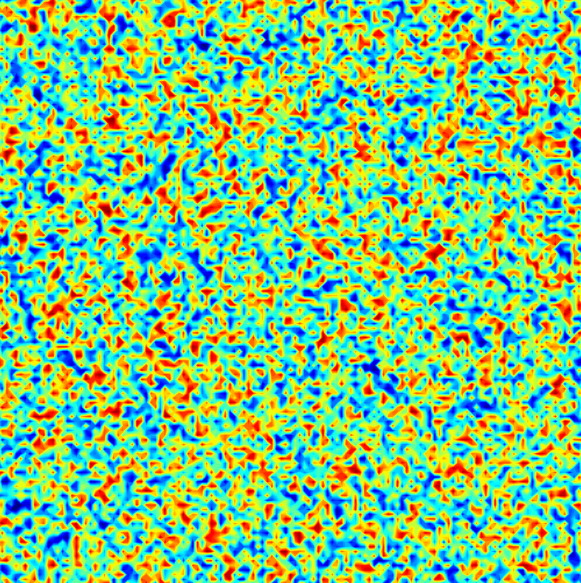}}\hfil
\subfloat[$t=19.98$\label{subfig:fig8j}]{\includegraphics[width=0.18\linewidth]{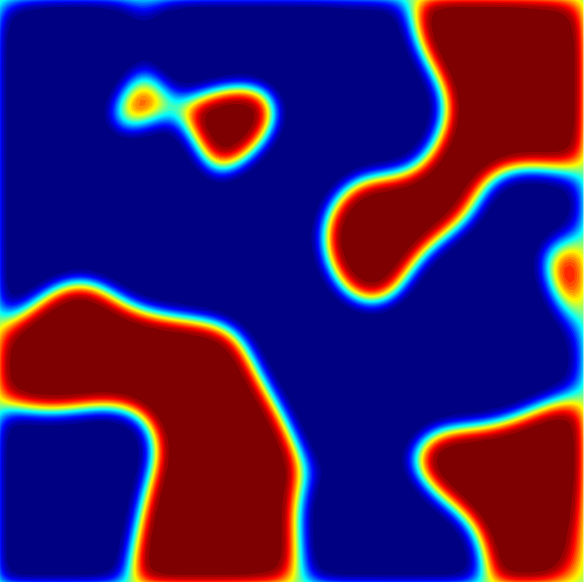}}\hfil
\subfloat[$t=40.02$\label{subfig:fig8k}]{\includegraphics[width=0.18\linewidth]{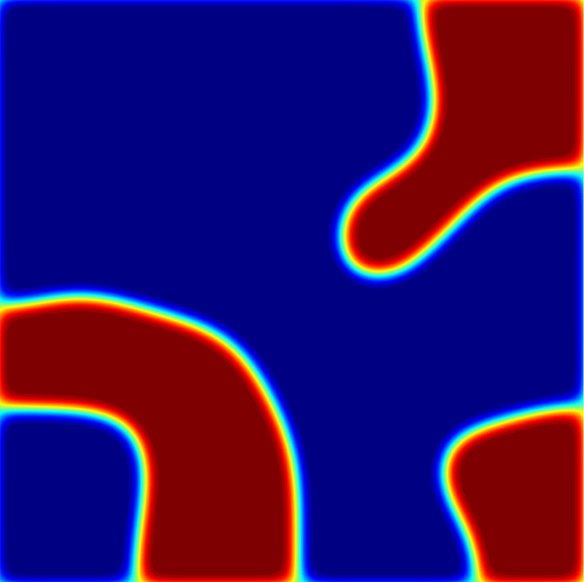}}\hfil
\subfloat[$t=60$\label{subfig:fig8l}]{\includegraphics[width=0.18\linewidth]{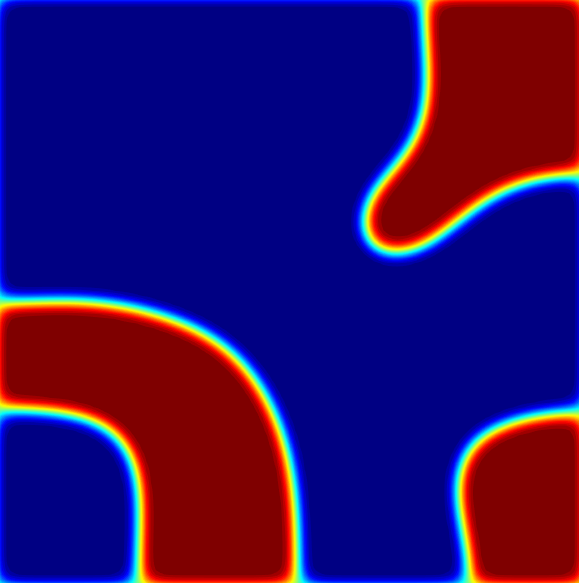}}
\caption{The dynamics of the interface in spinodal decomposition phenomena for  (i) $\alpha=1.3$ (first row) (ii) $\alpha=1.5$ (second row) (iii) $\alpha=1.8$ (third row)}
\label{fig8:main}
\end{figure}

\section{Conclusions}\label{sec:end}

We have considered an initial boundary value problem of the SFAC equation with logarithmic Flory-Huggins potential, which has been approximated by first-order weighted and shifted Gr\"unwald difference formulae. We provided localization and distribution results for the underlying coefficient and preconditioned matrix-sequences, where the design of proper preconditioners was based on the spectral information. Numerical results have been presented and have been critically discussed. In particular, in terms of convergence speed of preconditioned conjugate gradient method, we observe optimality and this is better than expected because the matrix structures are multilevel and because of the theoretical barriers proven in \cite{nega1,nega2}. For understanding this phenomenon, as future work we plan to study
\begin{itemize}
\item the number of the outliers as a function of the fineness parameters,
\item their asymptotic behavior and in particular their asymptotic separation from zero. 
\end{itemize}
Finally, given its importance in terms of modeling real-world phenomena, a further direction will be the analysis of the tempered SFAC equation, which generalizes the problem considered in the present work (see \cite{ahmad2025smoothing} and references therein).

\bibliographystyle{ieeetr}
\bibliography{Bibtex}
\end{document}